\documentclass[10pt,reqno,a4paper]{amsart}

\usepackage{amsfonts,amssymb}
\usepackage{pgfplots,tikz,graphicx,float}
\usepackage{epsfig}

\usepackage[hmargin=2.8cm,bmargin=2.8cm,tmargin=3cm]{geometry}

\title{Disjointness-preserving operators and isospectral Laplacians}

\author{Wolfgang~Arendt}
\author{James~B.~Kennedy}

\address{Wolfgang~Arendt, Institut f\"ur Angewandte Analysis, Universit\"at Ulm, Helmholtzstr.~18, D-89069 Ulm, Germany}
\email{wolfgang.arendt@uni-ulm.de}

\address{James~B.~Kennedy, Grupo de F\'isica Matem\'atica, Faculdade de Ci\^encias, Universidade de Lisboa, Campo Grande, Edif\'icio~C6, P-1749-016 Lisboa, Portugal}
\email{jbkennedy@fc.ul.pt}

\newtheorem{theorem}{Theorem}[section]
\newtheorem{lemma}[theorem]{Lemma}
\newtheorem{proposition}[theorem]{Proposition}
\newtheorem{corollary}[theorem]{Corollary}

\theoremstyle{remark}
\newtheorem{remark}[theorem]{Remark}
\newtheorem{definition}[theorem]{Definition}
\newtheorem{example}[theorem]{Example}

\numberwithin{equation}{section}
\numberwithin{figure}{section}

\newcommand{\R}{\mathbb{R}}
\newcommand{\N}{\mathbb{N}}
\newcommand{\C}{\mathbb{C}}
\newcommand{\Z}{\mathbb{Z}}
\newcommand{\DLL}{\Delta^D_{L^2 (\Omega)}}
\newcommand{\NLL}{\Delta^N_{L^2 (\Omega)}}
\newcommand{\DLC}{\Delta^D_{C_0 (\Omega)}}
\newcommand{\NLC}{\Delta^N_{C (\overline{\Omega})}}

\DeclareMathOperator{\dist}{dist}
\DeclareMathOperator{\supp}{supp}
\DeclareMathOperator{\capacity}{cap}
\DeclareMathOperator{\interior}{int}
\DeclareMathOperator{\diam}{diam}

\newcommand*{\gre}{>\!\!\!>}

\begin{document}

\begin{abstract}
All the known counterexamples to Kac' famous question ``can one hear the shape of a drum'', i.e., does isospectrality of two Laplacians on domains imply that the domains are congruent, consist of pairs of domains composed of copies of isometric building blocks arranged in different ways, such that the unitary operator intertwining the Laplacians acts as a sum of overlapping ``local'' isometries mapping the copies to each other.

We prove and explore a complementary positive statement: if an operator intertwining two appropriate realisations of the Laplacian on a pair of domains preserves disjoint supports, then under additional assumptions on it generally far weaker than unitarity, the domains are congruent. We show this in particular for the Dirichlet, Neumann and Robin Laplacians on spaces of continuous functions and on $L^2$-spaces.
\end{abstract}



\thanks{The work of the second-named author was supported by the Funda{\c{c}}{\~a}o para a Ci{\^e}ncia e a Tecnologia, Portugal, via the program ``Investigador FCT'', reference IF/01461/2015, and project PTDC/MAT-CAL/4334/2014, as well as a fellowship of the Alexander von Humboldt Foundation, Germany. The second-named author would also like to express his thanks for the generous hospitality afforded to him during very pleasant visits to Ulm University, where part of the paper was written. Both authors extend their thanks to Moritz Gerlach for providing the domains in Figure~\ref{fig:isospectral-pair}, which first appeared in \cite[Figure~1]{arendt14} in a slightly different form.}

\date{\today}

\maketitle

\section{Introduction}
\label{sec:intro}

One cannot hear the shape of a drum. More than 25 years have elapsed since Gordon, Webb and Wolpert \cite{gordon92} answered Kac' famous question in the negative: there exist pairs of planar domains $\Omega_1, \Omega_2 \subset \R^2$ which are not congruent to each other, but whose Dirichlet or Neumann Laplacians have the same spectra and thus which ``sound the same''; see \cite{berard93,buser94,chapman95,giraud10} for further expositions.

On the other hand, since Kac asked his question in the 1960s \cite{kac66} a huge body of literature has developed around proving positive answers within certain special classes of domains or manifolds, or for certain properties weaker than congruence; we refer to the survey \cite{datchev13}, the recent introduction \cite{lu15}, and mention other interesting contributions from the last few years \cite{grieser13,hezari12,lu16,zelditch09}, as well as the references therein, for a glimpse into the current state of affairs. A typical approach, following a broad scheme set out by Kac himself, is to extract information from formulae such as heat and wave traces which relate normalised sums or other combinations of eigenvalues to geometric properties of the domain.

Let us start with something simpler: what would a general positive answer look like? Suppose $\Omega_1$ and $\Omega_2$ are isospectral for some realisation $\Delta_{\Omega_i}$ of the Laplacian, $i=1,2$ (say, in the most natural setting of $L^2$-spaces). Then the canonical unitary operator $U: L^2(\Omega_1) \to L^2 (\Omega_2)$ defined by mapping the $n$th normalised eigenfunction $\psi_n$ of $-\Delta_{\Omega_1}$ on $\Omega_1$ to its counterpart $\varphi_n$ on $\Omega_2$,
\begin{displaymath}
	U\psi_n = \varphi_n\qquad \text{for all } n \geq 1,
\end{displaymath}
which is well defined for any pair of domains, has the additional property that it \emph{intertwines} the respective Laplacians, as well as the corresponding heat semigroups: if $f$ is in the domain $D(\Delta_{\Omega_1})$ of $\Delta_{\Omega_1}$, then $Uf \in D(\Delta_{\Omega_2})$, and
\begin{displaymath}
	U (\Delta_{\Omega_1} f) = \Delta_{\Omega_2} (Uf) \qquad \text{for all } f \in D(\Delta_{\Omega_1});
\end{displaymath}
in terms of the semigroups $(e^{t\Delta_{\Omega_i}})_{t\geq 0}$, this reads
\begin{equation}
\label{eq:semigroup-intertwining}
	U (e^{t \Delta_{\Omega_1}}f) = e^{t \Delta_{\Omega_i}} (Uf) \qquad \text{for all } f\in L^2(\Omega_1) \text{ and all } t>0.
\end{equation}
A corresponding statement must also hold \emph{mutatis mutandis} for $U^{-1}$. We also see immediately that the existence of such an intertwining operator $U$ is \emph{equivalent} to the isospectrality of the domains.

Thus Kac' question is equivalent to asking whether the existence of a unitary operator intertwining the Laplacians implies that the domains are congruent; in other words, whether this mapping $U$ must take the form
\begin{equation}
\label{eq:wuensch-form}
	Uf = f \circ \tau
\end{equation}
for all $f$ in the domain of the Laplacian on $L^2(\Omega_1)$ (and hence for all $f \in L^2 (\Omega_1)$, by density), for some isometry $\tau : \R^d \to \R^d$ for which $\tau (\Omega_2) = \tau (\Omega_1)$.\footnote{There is a slight technicality here: if $\Omega_1$ is an open set and $\Omega_2$ differs from it by a \emph{set of capacity zero}, then the Laplacians on $L^2 (\Omega_1)$ and $L^2 (\Omega_2)$ coincide without the sets being congruent; see for example \cite[Section~3]{arendt02}. We may ignore this by always selecting, among the equivalence class of all open sets differing by a set of capacity zero, the one which is regular in capacity.}

While the answer is clearly no, our point of departure is the observation that in all the counterexamples found two decades ago subsequent to the work of Gordon {\it et al} \cite{berard93,buser94,chapman95,gordon92}, the explicit intertwining operator $U$ mapping eigenfunctions on $\Omega_1$ to eigenfunctions on $\Omega_2$ always has a very special form: it is always a \emph{finite sum of overlapping ``local'' isometries}, each locally taking the form \eqref{eq:wuensch-form}. Let us explain this using the principal example (the ``propellers'') of Buser \emph{et al} \cite[Section 2]{buser94}, which was revisited in \cite{arendt14}. We start out with a basic building block, a triangle $T$; in the words of B\'erard \cite{berard93} this is the \emph{brique fondamentale}.
\begin{figure}[H]
\begin{center}
\begin{tikzpicture}[scale=0.7]
\draw[very thick] (0,0) -- (5,0);
\draw[very thick,densely dashed] (0,0) -- (1.8,2.4);
\draw[very thick,dotted] (1.8,2.4) -- (5,0);
\end{tikzpicture}
\caption{The triangle $T$.}
\label{fig:triangle}
\end{center}
\end{figure}
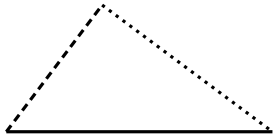
By reflecting the triangle in two different ways as in Figure~\ref{fig:isospectral-pair}, we obtain two domains $\Omega_1$ and $\Omega_2$, each composed of seven copies of $T$ as indicated; we will call these copies $T_1,\ldots,T_7$, abbreviated to $1,\ldots,7$ in the figure.
\begin{figure}[ht!]     
\vspace*{3mm} 
\centering
\epsfig{file=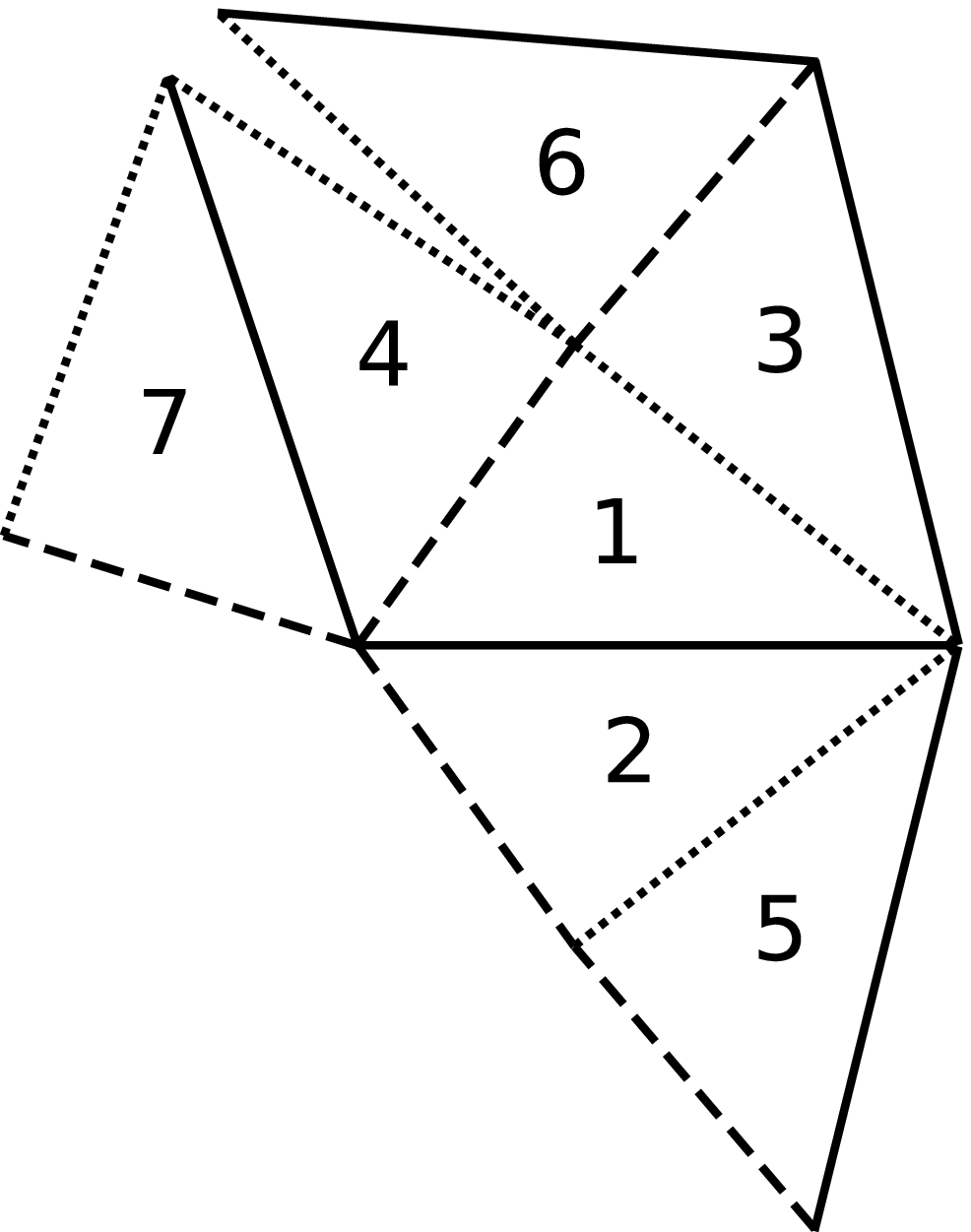, height=45mm} \hspace{20mm}
\epsfig{file=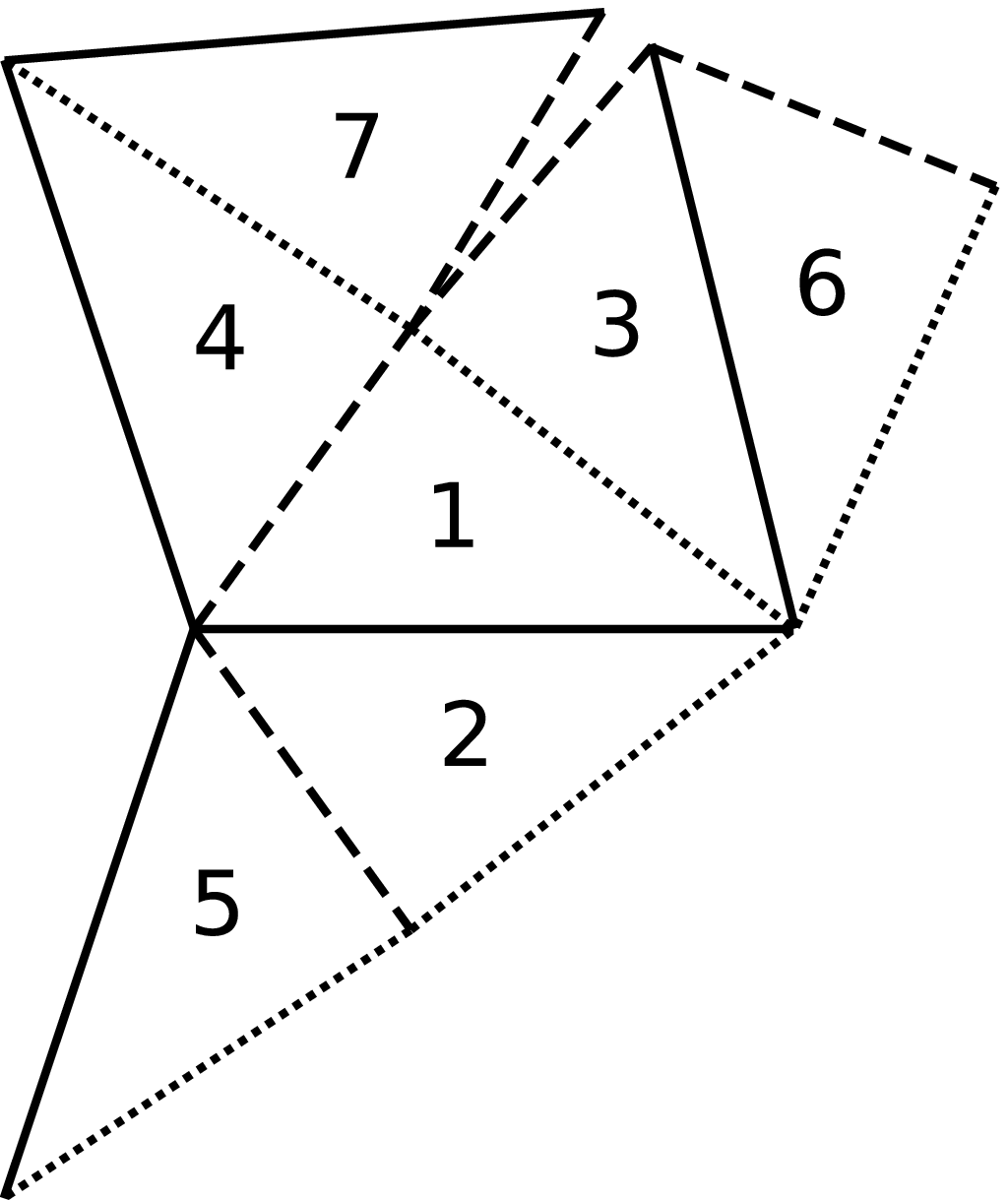, height=45mm}
\mbox{}
\makebox[0pt]{\raisebox{10mm}[0pt][0pt]{\mbox{}\hspace*{-180mm}$\Omega_1$ }}
\makebox[0pt]{\raisebox{10mm}[0pt][0pt]{\mbox{}\hspace*{-17mm}$\Omega_2$ }}
\caption{Two isospectral domains composed of seven isometric triangles, as used in \cite{arendt14} and based on the ``warped propeller'' domains of \cite{buser94}.}
\label{fig:isospectral-pair}
\end{figure}
If $T$ is a general scalene triangle, then $\Omega_1$ and $\Omega_2$ are not congruent. Nevertheless, there is a unitary operator $U:L^2 (\Omega_1) \to L^2 (\Omega_2)$ intertwining the Dirichlet Laplacians on $L^2(\Omega_1)$ and $L^2(\Omega_2)$. 
We identify each of the spaces $L^2 (\Omega_i)$, $i=1,2$, with $L^2(T)^7$ in a different way, by writing $f = (f_1,\ldots,f_7)^T \in L^2 (\Omega_i)$ if $f_j \in L^2(T_j)$, $j=1,\ldots,7$. Then $U$ is given locally as a sum of isometries: for each triangle subset of $\Omega_2$, $U$ is given by mapping three given triangles of $\Omega_1$ onto it, i.e., up to a normalising constant, $U$ has the form $(Uf)_i = (-1)^{j_1} f_{i_1} + (-1)^{j_2} f_{i_2} + (-1)^{j_3} f_{i_3}$ for each $i=1,\ldots,7$, for appropriate numbers $i_1,i_2,i_3\in\{1,\ldots,7\}$ and appropriate signs $j_1(i),j_2(i),j_3(i) \in \{-1,1\}$, which are chosen in such a way that $H^1_0 (\Omega_1)$ is mapped into $H^1_0 (\Omega_2)$; cf.\ \cite[Section 6]{arendt14}, \cite[Section 2]{buser94} or \cite[Section II]{giraud10} for a more detailed explanation-cum-proof. The same principle holds for all the other examples, for both the Dirichlet and Neumann Laplacians, known from \cite{berard93,buser94,chapman95,gordon92} etc.; see \cite[Section IV]{giraud10} and also the classifications in \cite{okada01,schillewaert11}.

In the present contribution, we will explore a slight variant of Kac' inverse problem, which seems quite natural in light of the above observation, and which was already raised in a similar form in the introduction to \cite{arendt14}. Namely, we ask what happens if one expressly prohibits such an overlapping: we are interested in operators $U$ intertwining Laplacians on two domains which have the additional property of \emph{preserving disjoint supports}: if $f,g$ are functions defined on $\Omega_1$ such that
\begin{displaymath}
	f\cdot g = 0
\end{displaymath}
pointwise everywhere (or almost everywhere) on $\Omega_1$, then we demand that
\begin{displaymath}
	Uf \cdot Ug = 0
\end{displaymath}
on $\Omega_2$. We will see that a unitary operator that intertwines the Dirichlet Laplacians on $L^2$ on two different domains and which is additionally disjointness-preserving already forces the domains to be congruent (see Corollary~\ref{cor:dirichlet-l2-kac-disj}). But now there is no particular reason to assume unitarity: we shall attempt to explore systematically the question of whether, and under what circumstances, a (general) disjointness-preserving intertwining operator, which is not necessarily unitary, is in fact given by an isometry. Of course, if $U$ is not unitary, then it does not necessarily have norm one, and so its ``desired'' form becomes not \eqref{eq:wuensch-form} but more generally
\begin{equation}
\label{eq:allgemeine-wuensch-form}
	Uf = cf\circ \tau
\end{equation}
for an isometry $\tau : \Omega_2 \to \Omega_1$ and a nonzero constant $c \in \C$. A further consequence of dropping unitarity is that we are no longer restricted to $L^2$-spaces (cf.\ Proposition~\ref{prop:unitary}).

We will begin by defining the realisations of the Laplacian which are relevant for us -- the Dirichlet and Neumann Laplacians, both on spaces of continuous functions (easier and more natural to work with) and on $L^2$ spaces (more natural from the point of view of the abstract theory) -- in Section~\ref{sec:notation}, where we will also define precisely what we mean by ``disjointness-preserving'' and ``intertwining operators''. In Section~\ref{sec:disjointness-preserving} we will present two introductory results, which show the effect of these two assumptions separately: roughly speaking, under the right technical assumptions, a disjointness-preserving operator $U$ (on spaces of continuous functions, say) always has the form $Uf(y) = h(y) f(\tau (y))$, $y \in \Omega_2$, for locally continuous functions $h:\Omega_2 \to \C$ and $\tau: \Omega_2 \to \Omega_1$; see Lemma~\ref{lem:disjointness-form}. If $U$ has this form and additionally intertwines the Laplacians (even just on the set of test functions), then $h$ is forced to be constant and $\tau$ an isometry -- but, again, this is only ``local'' in a certain sense; see Lemma~\ref{lem:intertwining-form}.

In order to achieve ``global'' results, more is needed, as simple examples at the end of Section~\ref{sec:disjointness-preserving} show. So, in the subsequent sections, we consider concrete realisations of the Dirichlet and Neumann Laplacians more carefully: we give conditions under which a disjointness-preserving operator $U$ intertwining Dirichlet Laplacians defined on the space $C_0$ (the closure of $C_c^\infty$ with respect to the sup norm $\|\cdot\|_\infty$; see Definition~\ref{def:c0}) has the form \eqref{eq:allgemeine-wuensch-form}, and in particular $\Omega_1$ and $\Omega_2$ are congruent, in Section~\ref{sec:dirichlet}: see Theorems~\ref{thm:disj-pres-dirichlet-char-c0} and~\ref{thm:disj-pres-dirichlet-c0-isom} and Corollary~\ref{cor:dirichlet-c0-kac-disj}. Corresponding results for the Neumann and Robin Laplacians are obtained in Section~\ref{sec:neumann}; see Theorems~\ref{thm:disj-pres-neumann-char-c}, \ref{thm:disj-pres-neumann-c-isom} and~\ref{thm:disj-pres-neumann-c-c1-isom} and Corollary~\ref{cor:neumann-c0-kac-disj} in particular. Let us state a theorem which summarises our results from Section~\ref{sec:dirichlet} in somewhat simplified form and under somewhat stronger assumptions.

\begin{theorem}
\label{thm:summary}
Suppose $\Omega_1, \Omega_2 \subset \R^d$ are bounded, Lipschitz domains such that $\Omega_1$ is connected and suppose $U : C_0 (\Omega_1) \to C_0 (\Omega_2)$ is a bounded linear operator satisfying
\begin{itemize}
\item[(a)] the disjointness-preserving condition ``$f\cdot g = 0$ implies $(Uf)\cdot (Ug)=0$, for all $f,g\in C_0 (\Omega_1)$'', and
\item[(b)] the intertwining property $U(\Delta f) = \Delta (Uf)$ in the sense of distributions, for all $f\in C_c^\infty (\Omega_1)$.
\end{itemize}
If in addition at least one of the following conditions is satisfied
\begin{itemize}
\item[(1)] $|\Omega_1|=|\Omega_2|$, or
\item[(2)] $U$ has dense range, or
\item[(3)] $\Omega_2$ is connected and the Dirichlet Laplacians on $\Omega_1$ and $\Omega_2$ have the same first eigenvalue,
\end{itemize}
then $\Omega_1$ and $\Omega_2$ are congruent and $U$ has the form \eqref{eq:allgemeine-wuensch-form}.
\end{theorem}

We return to considering the Dirichlet Laplacian on $L^2$-spaces in Section~\ref{sec:dirichlet-l2}, where we obtain similar results; although somewhat different techniques are required, see Theorem~\ref{thm:disj-pres-dirichlet-l2-isom}; and finish by showing that a disjointness-preserving unitary operator which intertwines Dirichlet, Neumann or Robin Laplacians on $L^2$ does in fact have the form \eqref{eq:wuensch-form}, in Corollary~\ref{cor:dirichlet-l2-kac-disj} and Theorem~\ref{thm:neumann-robin-l2-kac-disj}.

Thus we obtain a positive answer to versions of Kac' question in various settings under the additional assumption that our intertwining operators are disjointness-preserving, complementing the counterexamples described above. We also take this opportunity to recall that Kac' problem seems to be completely open for the Robin Laplacians: there are no known pairs of domains which are isospectral for the Robin Laplacian, say, for a common boundary constant $\beta \neq 0$ (see the discussion in \cite{arendt14}).

While motivated in large part by the above-mentioned observation about the form of the known counterexamples, the current note also follows in the tradition of earlier works of one of the present authors \cite{arendt01,arendt02,arendt12}, which have also been extended recently in other directions \cite{keller15,lenz18}. In these works, as here, one investigates operators having an intertwining property such as \eqref{eq:semigroup-intertwining} but where unitarity is replaced by some other property. For example, in \cite{arendt02}, it is shown that if there is an operator $U$ which intertwines Dirichlet, Neumann or Robin Laplacians and which is an \emph{order isomorphism} on $L^2$ (i.e., $U: L^2 (\Omega_1) \to L^2 (\Omega_2)$ is linear, bijective and $Uf \geq 0$ almost everywhere in $\Omega_2$ if and only if $f \geq 0$ almost everywhere in $\Omega_1$), then $\Omega_1$ and $\Omega_2$ are congruent, and up to a multiplicative constant $U$ has the form \eqref{eq:wuensch-form}. This can be extended to $p\neq 2$ if one replaces order isomorphism with \emph{isometric isomorphism}. Since under these assumptions the intertwining order isomorphism maps positive solutions to positive solutions of the associated heat equation, this may be interpreted as saying that ``diffusion determines the domain''.

These results are extended to smooth manifolds in \cite{arendt12}, while \cite{arendt01} deals with the corresponding question on the space of continuous functions vanishing at the boundary and at infinity. In \cite{keller15}, a similar analysis was given on weighted discrete graphs; this was extended recently in \cite{lenz18} to the much more general abstract setting of pairs of Dirichlet forms intertwined by order isomorphisms on $L^2$-spaces under a wide variety of assumptions.

We will draw certain techniques and some background results from these works, in particular \cite{arendt01,arendt02}. However, disjointness preservation is a much weaker property than that of being an order isomorphism; indeed, the latter is easily seen to imply the former. Moreover, in light of the nature of the known counterexamples, the former is also arguably more natural in the context of isospectrality. Finally, most our principal results (Theorems~\ref{thm:disj-pres-dirichlet-char-c0} and~\ref{thm:disj-pres-neumann-char-c}, and their respective extensions and corollaries) do not actually require our operator $U$ to intertwine the Laplacians: it merely has to have this property on the much smaller space of test functions $C_c^\infty (\Omega_1)$. This is more than just a technicality: this space is not a core for the Dirichlet or Neumann Laplacians; indeed, these are different self-adjoint extensions of the Laplacian on $C_c^\infty (\Omega_1)$. Thus this property does not imply that the actual Dirichlet, Neumann or Robin Laplacians are intertwined; indeed, it suggests that the congruence of the domains is appearing at a much more fundamental level.

%

\section{Notation and Definitions}
\label{sec:notation}

Let $\Omega \subset \R^d$, $d\geq 1$, be an open set which will be fixed throughout this section. We start out by defining the realisations of the Laplacian in which we will be interested. We will consider a total of four: the Laplacians with Dirichlet and Neumann boundary conditions, being realised either on the space $L^2 (\Omega)$ or on an appropriate space of continuous functions. We start with the function spaces we will need.

\begin{definition}
\label{def:c0}
Suppose $\Omega \subset \R^d$ is an open set. 
\begin{itemize}
\item[(a)] We set
\begin{displaymath}
	C_c^\infty (\Omega) := \{ f|_\Omega: f\in C^\infty (\R^d) \text{ and $\supp f$ is compactly contained in } \Omega \},
\end{displaymath}
where $\supp f \subset \R^d$ is the support of $f$, i.e., the closure in $\R^d$ of the set $\{x \in \R^d: f(x)\neq 0\}$.
\item[(b)] We define the space $C_0 (\Omega)$ to be the closure of $C_c^\infty (\Omega)$ with respect to the supremum norm $\|\cdot\|_\infty$, i.e.\ $\|u\|_\infty = \sup_{x\in \Omega} |u(x)|$.
\item[(c)] We set $C_b (\Omega)$ to be the space of bounded and continuous functions on $\Omega$, equipped with the supremum norm $\|\cdot\|_\infty$.
\item[(d)] The space $L^2 (\Omega)$ is the Hilbert space of square integrable Lebesgue measurable functions on $\Omega$, equipped with the usual inner product $\langle\,\cdot\,,\,\cdot\,\rangle$.
\item[(e)] The Sobolev space $H^1 (\Omega)$ is the Hilbert space of $L^2 (\Omega)$-functions whose distributional partial derivatives all lie in $L^2 (\Omega)$; this space will also be equipped with any of the usual equivalent inner products.
\item[(f)] The space $H^1_0 (\Omega)$ is the closure of $C_c^\infty (\Omega)$ with respect to any one of the equivalent $H^1$-norms.
\end{itemize}
\end{definition}

We observe that functions $u \in C_0 (\Omega)$ are, by construction, continuous on $\overline{\Omega}$ and pointwise zero on $\partial\Omega$, and satisfy $\lim_{|x|\to \infty} u(x) = 0$. In particular, $C_0 (\Omega) \subset C_b (\Omega)$. In terms of our operators, we will first consider the $L^2$-case, corresponding to the usual weak formulation.

\begin{definition}
\label{def:dirichlet-laplacian-l2}
Suppose $\Omega \subset \R^d$ is an open set. We define the \emph{Dirichlet Laplacian on $L^2 (\Omega)$}, which we shall denote by $-\DLL$, to be the operator on $L^2(\Omega)$ associated with the sesquilinear form
\begin{equation}
\label{eq:dirichlet-form}
	a(u,v) := \int_\Omega \nabla u \cdot \overline{\nabla v}\,\textrm{d}x,
\end{equation}
for $u,v \in H^1_0 (\Omega)$. That is, $\DLL$ is given by
\begin{displaymath}
\begin{aligned}
	D (\DLL) &= \{ u \in H^1_0 (\Omega): \exists f \in L^2 (\Omega) \text{ s.t.\ } a(u,v)=\langle f,v \rangle 
	\text{ for all } v \in H^1_0 (\Omega) \},\\
	\DLL u &= -f.
\end{aligned}
\end{displaymath}
\end{definition}

We shall also write just $\Delta^D$ or $\Delta^D_\Omega$ if there is no danger of confusion, and observe that the choice of the space $H^1_0 (\Omega)$ encodes the boundary condition in the usual weak sense. It is a routine exercise to show that $\DLL$ is also given by
\begin{displaymath}
\begin{aligned}
	D (\DLL) &= \{ u \in H^1_0 (\Omega): \Delta u \in L^2 (\Omega) \},\\
	\DLL u &= \Delta u,
\end{aligned}
\end{displaymath}
where $\Delta f$ is interpreted in the distributional sense if $f \in L^2 (\Omega)$.

We next consider the corresponding operator on spaces of continuous functions.

\begin{definition}
\label{def:dirichlet-laplacian-c0}
Suppose $\Omega \subset \R^d$ is an open set. The \emph{Dirichlet Laplacian on $C_0 (\Omega)$}, denoted by $-\DLC$, is defined by
\begin{displaymath}
\begin{aligned}
	D (\DLC) &= \{ u \in C_0 (\Omega): \Delta u \in C_0 (\Omega) \},\\
	\DLC u &= \Delta u,
\end{aligned}
\end{displaymath}
where $\Delta u$ is again to be understood in the distributional sense.
\end{definition}

If it is clear that we are in $C_0$, then we shall sometimes write $\Delta^D_\Omega$ or just $\Delta^D$ for this operator. If $\Omega$ is bounded, then $C_0 (\Omega) \subset L^2(\Omega)$ and in fact $\DLC$ is the \emph{part of $\DLL$ in $C_0 (\Omega)$}. This also means that $D(\DLC) \subset H^1_0(\Omega)$; see \cite{arendt01}.

If $\Omega$ has finite measure then the operator $-\DLL$ has compact resolvent and hence a sequence of eigenvalues of the form
\begin{displaymath}
	0 < \lambda_1 (-\DLL) \leq \lambda_2 (-\DLL) \leq \ldots
\end{displaymath}
constituting the entirety of the spectrum $\sigma (-\DLL)$, where each eigenvalue is repeated according to its finite multiplicity (noting that algebraic and geometric multiplicities are always equal) and the associated eigenfunctions may be chosen to form an orthonormal basis of $L^2 (\Omega)$.

\begin{remark}
\label{rem:dirichlet-regular}
A bounded open set $\Omega \subset \R^d$ is called \emph{Dirichlet} (or \emph{Wiener}) regular if for each $g \in C(\partial\Omega)$ there exists a function $u \in C^2(\Omega) \cap C(\overline{\Omega})$ such that $\Delta u = 0$ in $\Omega$ and $u|_{\partial\Omega} = g$. It turns out that the bounded open set $\Omega$ is Dirichlet regular if and only if the first eigenfunction of $\DLL$ is in $C_0 (\Omega)$; in this case, all eigenfunctions are in $C_0 (\Omega)$ and the spectra of $\DLL$ and $\DLC$ coincide. In this case, we will just write $\lambda_k^D (\Omega)$ for these eigenvalues. Moreover, the resolvent set of $\DLC$ is non-empty if and only if $\Omega$ is Dirichlet regular, and in this case $\DLC$ generates a holomorphic $C_0$-semigroup on $C_0(\Omega)$. See \cite{arendt99} for more details and further information.
\end{remark}

Similar assertions hold in the case of Neumann and Robin boundary conditions. In the $L^2$-setting, we replace the space $H^1_0 (\Omega)$ with $H^1 (\Omega)$ and for a bounded measurable function $\beta \in L^\infty (\partial\Omega)$ defined on $\partial\Omega$ we introduce the form
\begin{equation}
\label{eq:robin-form}
	a_\beta (u,v) := \int_\Omega \nabla u \cdot \overline{\nabla v}\,\textrm{d}x + \int_{\partial\Omega} \beta uv\,\textrm{d}\sigma
\end{equation}
for $u,v \in H^1(\Omega)$, where $\sigma$ is surface measure on $\partial\Omega$. Clearly, the forms $a$ and $a_0$ agree. We then define the operators associated with the forms $a=a_0$ and $a_\beta$ on $H^1(\Omega)$ as follows.

\begin{definition}
\label{def:neumann-laplacian-l2}
Suppose $\Omega \subset \R^d$ is an open set.
\begin{itemize}
\item[(a)] For a function $u \in H^1 (\Omega)$, we define its distributional outer normal derivative $\frac{\partial u}{\partial \nu}$ to be the unique function $h \in L^2 (\partial\Omega)$, if one exists, such that
\begin{displaymath}
	\int_\Omega \nabla u \cdot \overline{\nabla v} + \Delta u \overline{v}\,\textrm{d}\sigma
	= \int_{\partial\Omega} h\overline{v}\,\textrm{d}x
\end{displaymath}
for all $v \in H^1(\Omega)$.
\item[(b)] The \emph{Neumann Laplacian on $L^2 (\Omega)$}, $-\NLL$, is defined by
\begin{displaymath}
\begin{aligned}
	D (\NLL) &= \{ u \in H^1 (\Omega): \exists f \in L^2 (\Omega) \text{ s.t.\ } a(u,v)=\langle f,v \rangle 
	\text{ for all } v \in H^1 (\Omega) \}\\
	&= \left\{ u \in H^1 (\Omega): \Delta u \in L^2(\Omega),\,\frac{\partial u}{\partial \nu} \text{ exists in $L^2(\Omega)$ and} =0 \right\},\\
	\NLL u &= f = \Delta u,
\end{aligned}
\end{displaymath}
where $a$ is given by \eqref{eq:dirichlet-form}.
\item[(c)] Given a function $\beta \in L^\infty (\partial\Omega)$, we define the \emph{Robin Laplacian associated with $\beta$ on $L^2(\Omega)$}, $-\Delta_{L^2(\Omega)}^\beta$, by
\begin{displaymath}
\begin{aligned}
	D (-\Delta_{L^2(\Omega)}^\beta) &= \{ u \in H^1 (\Omega): \exists f \in L^2 (\Omega) \text{ s.t.\ } a_\beta(u,v)=\langle f,v \rangle 
	\text{ for all } v \in H^1 (\Omega) \}\\
	&= \left\{ u \in H^1 (\Omega): \Delta u \in L^2(\Omega),\,\frac{\partial u}{\partial \nu} \text{ exists in $L^2(\Omega)$ and} = - \beta u \right\},\\
	-\Delta_{L^2(\Omega)}^\beta u &= f = \Delta u,
\end{aligned}
\end{displaymath}
where $a_\beta$ is given by \eqref{eq:robin-form}.
\end{itemize}
\end{definition}
The Neumann Laplacian clearly coincides with the Robin Laplacian when $\beta \equiv 0$, that is, $\NLL = \Delta_{L^2(\Omega)}^0$. Moreover, if it is clear which domain $\Omega$ we mean, and that we are in the $L^2$-setting, then we shall again simply write $\Delta^N$ and $\Delta^\beta$ for the Neumann and Robin Laplacians, respectively. If $\Omega$ satisfies a moderate regularity property, for example, if it is bounded and Lipschitz (i.e., $\partial\Omega$ is locally given by the graph of a Lipschitz continuous function), then $\NLL$ and $\Delta_{L^2(\Omega)}^\beta$ also have compact resolvent and their spectrum are of the same form as the spectrum of $\DLL$, namely
\begin{displaymath}
	0 = \lambda_1 (-\NLL) \leq \lambda_2 (-\NLL) \leq \ldots, \qquad \lambda_1 (-\Delta_{L^2(\Omega)}^\beta) \leq \lambda_2 (-\Delta_{L^2(\Omega)}^\beta) \leq \ldots,
\end{displaymath}
with the eigenvalues having the same properties as before; in particular, the eigenfunctions of each such operator may be chosen to form an orthonormal basis of $L^2(\Omega)$ and so on (see, e.g., \cite[Section~4.2]{bucur17}).

Finally, we wish to define a realisation of the Neumann and Robin Laplacians on spaces of continuous functions, i.e., $C(\overline{\Omega})$. Here we will always assume that $\Omega$ is bounded and Lipschitz, although many definitions can be given for more general domains; and we will also suppose that $\beta \in C(\partial\Omega)$. Under these assumptions, we shall consider \emph{the part of $\NLL$ in $C(\overline{\Omega})$} and \emph{the part of $\Delta_{L^2(\Omega)}^\beta$ in $C(\overline{\Omega})$} (as was done, for example, in \cite[Section~3]{warma06} and \cite{nittka11} for the Robin Laplacian under the assumption $\beta \geq \beta_0 > 0$; note however that the sign of $\beta$ does not enter into the construction, see, e.g., \cite{daners09a}).

\begin{definition}
\label{def:neumann-laplacian-c}
Suppose $\Omega \subset \R^d$ is a bounded open set with Lipschitz boundary.
\begin{itemize}
\item[(a)] The \emph{Neumann Laplacian on $C(\overline{\Omega})$}, $-\NLC$, is defined by
\begin{displaymath}
\begin{aligned}
	D (\NLC) &= \left\{ u \in H^1(\Omega) \cap C(\overline{\Omega}): \Delta u \in L^2(\Omega) \cap 
	C(\overline{\Omega}), \frac{\partial u}{\partial\nu} \in L^2 (\partial\Omega) \text{ and} =0 \right\},\\
	\NLC u &= \Delta u,
\end{aligned}
\end{displaymath}
where $\frac{\partial u}{\partial \nu}$ is as in Definition~\ref{def:neumann-laplacian-l2}.
\item[(b)] Let $\beta \in C(\partial\Omega)$. The \emph{Robin Laplacian on $C(\overline{\Omega})$}, $-\Delta_{C(\overline{\Omega})}^\beta$, is defined by
\begin{displaymath}
\begin{aligned}
	D (\Delta_{C(\overline{\Omega})}^\beta) &= \left\{ u \in H^1(\Omega) \cap C(\overline{\Omega}): \Delta u \in L^2(\Omega) \cap 
	C(\overline{\Omega}), \frac{\partial u}{\partial\nu} \in L^2 (\partial\Omega) \text{ and} = - \beta u \right\},\\
	\Delta_{C(\overline{\Omega})}^\beta u &= \Delta u.
\end{aligned}
\end{displaymath}
\end{itemize}
\end{definition}

If $\Omega$ is bounded and Lipschitz, then every eigenfunction of $\NLL$ and $\Delta_{L^2(\Omega)}^\beta$ is also in $C(\overline{\Omega})$ (every eigenfunction is certainly in $L^\infty (\Omega)$, see, e.g., \cite[Theorem~2.5]{daners09}; now \cite[Theorem~2.2]{warma06} or the arguments of \cite[Lemma~2.1]{bucur10} imply that they are also in $C(\overline{\Omega})$; this may also be deduced from \cite{nittka11}, where it is shown that $\Delta_{C(\overline{\Omega})}^\beta$ generates a holomorphic $C_0$-semigroup on $C(\overline{\Omega})$ for such $\Omega$). Hence the spectra of $\NLL$ and $\NLC$ coincide, as do the spectra of $\Delta_{L^2(\Omega)}^\beta$ and $\Delta_{C(\overline{\Omega})}^\beta$. In this case, we will write $\lambda_n^N (\Omega)$ and $\lambda_n^\beta (\Omega)$ for the corresponding Neumann and Robin eigenvalues, respectively.

We next introduce the two key notions with which we will be working: the notion of an \emph{intertwining operator}, and the notion of a \emph{disjointness-preserving operator}.

\begin{definition}
\label{def:intertwining}
Suppose $X_1$ and $X_2$ are Banach spaces, and $A_1: D(A_1) \subset X_1 \to X_1$ and $A_2: D(A_2) \subset X_2 \to X_2$ are linear operators. We say that $U : X_1 \to X_2$ \emph{intertwines} the operators $A_1$ and $A_2$ if
\begin{equation}
\label{eq:intertwining}
	x \in D(A_1) \implies Ux \in D(A_2) \text{ and } A_2 U x = U A_1 x.
\end{equation}
In this case we call $U$ an \emph{intertwining operator} (for $A_1$ and $A_2$).
\end{definition}

If $A_2$ is closed, a simple density argument shows that $U$ is intertwining whenever there exists a core\footnote{We recall that a core of an operator is a subset of its domain which is dense in that domain with respect to the operator norm.} $D$ of $A_1$ such that $UD \subset D(A_2)$ and $A_2 Ux = U A_1 x$ for all $x \in D$. Actually, we will often work with a weaker intertwining property, namely that \eqref{eq:intertwining} holds for a subset of the operator domain which is not necessarily a core.

\begin{remark}
\label{rem:intertwining}
If $A_j$ generates a $C_0$-semigroup $S_j$ on $X_j$, $j=1,2$, then a bounded linear operator $U: X_1 \to X_2$ intertwines $A_1$ and $A_2$ if and only if
\begin{displaymath}
	S_2(t)U = US_1(t) \qquad \text{for all } t\geq 0.
\end{displaymath}
\end{remark}

We next give an elementary result characterising unitary intertwining operators, which will be very useful in the sequel. It also gives us a natural analogue of them on spaces of continuous functions, where we can no longer talk about unitary operators.

\begin{proposition}
\label{prop:unitary}
Let $\Omega_1 \subset \R^{d_1}$ and $\Omega_2 \subset \R^{d_2}$ be bounded open sets and consider the Dirichlet Laplacians on $L^2 (\Omega_i)$, $i=1,2$. Denote by $\{ (\lambda_k^D (\Omega_i),\psi_k (\Omega_i))\}_{k=1}^\infty$ a sequence of eigenvalues and eigenfunctions forming an orthonormal basis of $L^2(\Omega_i)$, $i=1,2$. Then the following are equivalent.
\begin{itemize}
\item[(1)] $\lambda_k^D (\Omega_1) = \lambda_k^D (\Omega_2)$ for all $k\geq 1$;
\item[(2)] There exists a unitary intertwining operator $U: L^2 (\Omega_1) \to L^2(\Omega_2)$;
\item[(3)] There exists an invertible intertwining operator $U: L^2 (\Omega_1) \to L^2(\Omega_2)$.
\end{itemize}
This equivalence remains true for the Neumann and for the Robin Laplacians on $L^2$, if $\Omega_1$ and $\Omega_2$ are bounded and Lipschitz. Moreover, the implication (3) $\implies$ (1) continues to hold for the Dirichlet Laplacians on $C_0$ if $\Omega_1$ and $\Omega_2$ are Dirichlet regular.
\end{proposition}

\begin{proof}
(1) $\implies$ (2) Let $U$ be the unitary operator given by $U\psi_k (\Omega_1) = \psi_k (\Omega_2)$ for all $k\geq 1$. One sees, for example by the spectral theorem, that $U D(\Delta^D_{L^2(\Omega_1)}) = D(\Delta^D_{L^2(\Omega_2)})$ and $\Delta Uf = U\Delta f$ for all $f \in D(\Delta^D_{L^2(\Omega_1)})$.

(2) $\implies$ (3) Trivial.

(3) $\implies$ (1) Let $U: L^2 (\Omega_1) \to L^2 (\Omega_2)$  be invertible and intertwining; then $U D(\Delta^D_{L^2(\Omega_1)}) = D(\Delta^D_{L^2(\Omega_2)})$ and $\Delta Uf = U\Delta f$ for all $f \in D(\Delta^D_{L^2(\Omega_1)})$. Moreover, obviously also $U^{-1} D(\Delta^D_{L^2(\Omega_2)}) = D(\Delta^D_{L^2(\Omega_1)})$, and if $g \in D(\Delta^D_{L^2(\Omega_2)})$, say with $g=Uf$, then $U^{-1} \Delta g = U^{-1} \Delta Uf = U^{-1}U\Delta f = \Delta U^{-1}g$. Now let $\psi \in L^2 (\Omega_1)$ and $\lambda \in \R$. Then $\psi \in D(\Delta^D_{L^2(\Omega_1)})$ and $\Delta \psi = \lambda \psi$ if and only if $U\psi \in D(\Delta^D_{L^2(\Omega_2)})$ and $\Delta U\psi = \lambda U\psi$; note that $U\psi \neq 0$ since $U$ is invertible. The same is also true of $U^{-1}$. Thus $\lambda = \lambda_k (\Delta^D_{L^2(\Omega_1)})$ for some $k\geq 1$ if and only if $\lambda = \lambda_j (\Delta^D_{L^2 (\Omega_2)})$ for some $j\geq 1$. Since this holds for all $\lambda \in \R$, we see that $k=j$ in the case of simple eigenvalues, or correspondingly in the case of multiple eigenvalues the eigenspaces have the same dimension; and (1) holds. The argument is exactly for the Dirichlet Laplacians on $C_0$, and the Neumann and Robin Laplacians if $\Omega_1$ and $\Omega_2$ are bounded and Lipschitz.

Note that if a domain is Dirichlet regular, then its Dirichlet Laplacian spectra on $L^2$ and $C_0$ coincide (see \cite[Theorem~2.3]{arendt99}); in particular, if (3) is satisfied for the Dirichlet Laplacians on $C_0$, then in (1) it does not matter whether we consider the $L^2$- or the $C_0$-spectra.
\end{proof}

This is the setting of Kac' original question: does the existence of such a unitary intertwining operator imply that $\Omega_1$ and $\Omega_2$ are isospectral? Here, however, we will replace unitarity with the following property.

\begin{definition}
\label{def:disjointness-preserving}
Suppose $E_1$ and $E_2$ are Banach lattices. A bounded, linear operator $U: E_1 \to E_2$ is called \emph{disjointness-preserving} if
\begin{equation}
\label{eq:disjointness-preserving}
	|f| \wedge |g| = 0 \quad \implies \quad |Uf| \wedge |Ug| =0\qquad \text{for all } f,g \in E_1,
\end{equation}
where we use the notation $f \wedge g = \inf \{f,g\}$ in the sense of Banach lattices.
\end{definition}

In practice we will be interested (only) in the spaces $L^2(\Omega)$, $C_0(\Omega)$, $C_b(\Omega)$ and $C(\overline{\Omega})$, where $\Omega \subset \R^d$ is an open set, usually of finite Lebesgue measure; in the case $C(\overline{\Omega})$ we even restrict to bounded open sets. In these cases, \eqref{eq:disjointness-preserving} can be reformulated as
\begin{displaymath}
	f \cdot g = 0 \quad \implies \quad (Uf)\cdot (Ug)=0\qquad \text{for all } f,g,
\end{displaymath}
where the equalities should hold everywhere in $C$ or almost everywhere in $L^2$. For more on disjointness-preserving operators, we refer to \cite{abramovich92,pagter00,schep16} and the references therein.

\section{Disjointness-preserving operators}
\label{sec:disjointness-preserving}

In this section, we will present two key lemmata which show how the structural assumptions on $U$, namely that it be disjointness-preserving and that it intertwine Laplacians, force it to be at least \emph{locally} an isometry. Here we will work exclusively on spaces of continuous functions, as it is much easier to be able to work with point evaluations.

Our first lemma shows that any disjointness-preserving operator $U$ taking continuous functions on some open set $\omega_1$ to ones on another open set $\omega_2$ is, roughly speaking, locally of the form $Uf = hf\circ \tau$ for continuous maps $h: \omega_2 \to \C$ and $\tau : \omega_2 \to \omega_1$ (for more, general, properties of disjointness-preserving operators on Banach spaces, we refer to the volume \cite{abramovich92}; see also \cite{pagter00} and \cite{schep16}). The second lemma shows how the additional property of intertwining Laplacians then forces $h$ to be locally constant and $\tau$ to be locally an isometry.

\begin{lemma}
\label{lem:disjointness-form}
Suppose $\omega_1,\omega_2 \subset \R^d$ are open sets and $U\neq 0$ is a bounded linear mapping from $C_0 (\omega_1)$ into $C_b(\omega_2)$. Let
\begin{displaymath}
	\omega_2':=\{y \in \omega_2: \exists f \in C_0(\omega_1) \text{ such that } (Uf)(y) \neq 0\}.
\end{displaymath}
If $U$ satisfies the disjointness-preserving condition
\begin{equation}
\label{eq:disjointness-preserving-statement}
	f \cdot g = 0 \quad \implies \quad (Uf)\cdot (Ug)=0 \qquad \text{for all } f,g \in C_0 (\omega_1),
\end{equation}
then there exist functions $h: \omega_2' \to \C \setminus \{0\}$ and $\tau: \omega_2' \to \omega_1$ such that
\begin{equation}
\label{eq:disjointness-form}
	Uf(y) = h(y)f(\tau(y)) \qquad \text{for all } y \in \omega_2' \text{ and all } f \in C_0 (\omega_1).
\end{equation}
Moreover, $\omega_2' \neq \emptyset$ is open and $h$ and $\tau$ are continuous. If in addition $U(C_c^\infty (\omega_1)) \subset C^k (\omega_2)$ for some $0 \leq k \leq \infty$, then $h,\tau \in C^k (\omega_2')$.
\end{lemma}

\begin{remark}
\label{rem:disjointness-form-c}
While the choice of the space $C_0 (\omega_1)$ recalls the Dirichlet boundary condition, the same conclusion is obviously true if $U$ maps the whole of $C(\overline{\omega}_1)$ or $C_b(\omega_1)$ into $C_b(\omega_2)$. The space $C_0 (\omega_1)$ is simply the smallest of these spaces, and thus gives the weakest condition. If, for example, $U: C(\overline{\omega}_1) \to C(\overline{\omega}_2)$ is bounded and linear, and the disjointness-preserving condition \eqref{eq:disjointness-preserving-statement} holds on $C(\overline{\omega}_1)$, then so too does \eqref{eq:disjointness-form}, and with the same proof.
\end{remark}

\begin{proof}[Proof of Lemma~\ref{lem:disjointness-form}]
The idea of the proof is already contained in \cite[Proposition~2.4]{arendt02}, albeit under somewhat different assumptions. The set $\omega_2'$ is open since each $Uf$ is continuous, and non-empty since $U \neq 0$. Now suppose $y \in \omega_2'$. Then $\varphi_y :=Uf (y)$ defines a non-zero functional on $C_0 (\omega_1)$. We claim that the support of $\varphi_y$ is a singleton. Indeed, if $x_1,x_2 \in \supp \varphi_y$, $x_1 \neq x_2$, then by definition of the support of a functional there exist functions $f,g \in C_0 (\omega_1)$ with disjoint support (i.e., $f\cdot g= 0$ everywhere) such that $f(x_1) \neq 0$, $g(x_2)\neq 0$, $\varphi_y f \neq 0$ and $\varphi_y g \neq 0$. But by assumption
\begin{displaymath}
	0 = (Uf)(y)\cdot (Ug)(y) = \varphi_yf \cdot \varphi_y g,
\end{displaymath}
a contradiction. It follows that there exist $0 \neq h(y) \in \C$ and $\tau(y) \in \omega_1$ such that
\begin{displaymath}
	\varphi_y = h(y) \delta_{\tau(y)},
\end{displaymath}
where $\delta_{\tau(y)}$ is the delta distribution at the point $\tau(y)$. This means that
\begin{displaymath}
	Uf(y) = h(y)f(\tau(y)).
\end{displaymath}
(Note in particular that $\tau (y) \in \partial\omega_1$ is impossible since then $f(\tau(y))=0$, meaning $Uf(y)=0$ for all $f \in C_0(\omega_1)$.) Since $y \in \omega_2'$ was arbitrary, this also means that $h(\omega_2') \subset \C\setminus\{0\}$, and \eqref{eq:disjointness-form} holds for all $y \in \omega_2'$.

Finally, we prove the regularity of $\tau$ and $h$. Here the proof is essentially the one given in \cite[Proposition~2.4]{arendt02}. If $\tau$ is not continuous on $\omega_2'$, then we can find $y,y_n \in \omega_2'$ and $\varepsilon>0$ such that $y_n\to y$ but $|\tau (y_n) - \tau(y)| \geq \varepsilon$ for all $n$. If we choose $f \in C_0 (\omega_1)$ such that $f(\tau(y))=1$ and $\supp f \subset B_\varepsilon (\tau(y))$, then $f(\tau(y_n))=0$ for all $n$, meaning $Uf(y_n)=0$ for all $n$. But $Uf(y)=h(y) \neq 0$. This contradicts the continuity of $Uf$. Hence $\tau$ is continuous on $\omega_2'$.

Next, fix an arbitrary open set $\omega$ which is compactly contained in $\omega_2'$. Then $\tau (\overline{\omega}) \subset \omega_1$ is compact since $\tau$ is continuous. Choose $f \in C_c^\infty (\omega_1)$ such that $f|_{\tau(\overline\omega)} = 1$. Then $Uf=h$ on $\overline\omega$. In particular, $h \in C(\overline{\omega})$. If $Uf \in C^k$, $k \leq \infty$, then the same argument shows that $h \in C^k$.

Finally, writing $x = (x_1,\ldots, x_d) \in \R^d$ and similarly $\tau = (\tau_1,\ldots,\tau_d)$, if we choose $f \in C_c^\infty (\omega_1)$ such that $f(x)=x_j$ on $\tau (\overline{\omega})$ for some $j=1,\ldots,d$, then
\begin{displaymath}
	Uf(y) = h(y) \tau_j (y).
\end{displaymath}
Since $Uf,h\in C^k$, we therefore also have $\tau \in C^k (\overline{\omega})$.
\end{proof}

\begin{lemma}
\label{lem:intertwining-form}
Suppose $\omega_1,\omega_2 \subset \R^d$ are open sets and $h: \omega_2 \to \C \setminus \{0\}$ and $\tau: \omega_2 \to \omega_1$ are continuous. Define the linear mapping $U: C_c^\infty (\omega_1) \to C (\omega_2)$ by $Uf(y) = h(y) f(\tau(y))$ for all $y \in \omega_2$, and assume additionally that
\begin{equation}
\label{eq:intertwining-property}
	\Delta (Uf) = U (\Delta f) \qquad \text{for all } f \in C_c^\infty (\omega_1),
\end{equation}
where $\Delta (Uf)$ is understood as a distribution. Then, on each connected component $N$ of $\omega_2$, $h|_N$ is constant and $\tau|_N : N \to \tau (N) \subset \omega_1$ is an isometry.
\end{lemma}

Note in particular that the lemma also shows that $Uf \in C^\infty (\omega_2)$ and \eqref{eq:intertwining-property} in fact holds pointwise. Also observe that we do not actually need $U$ to map into $C(\omega_2)$: any $L^p (\omega_2)$-space could also be used, with the same proof.

The idea of Lemma~\ref{lem:intertwining-form} appeared implicitly in \cite[Steps (c)-(e) of the proof of Proposition~2.4]{arendt02} under stronger regularity assumptions, and here the formal argument is essentially the same; for the sake of completeness, we will reproduce the calculations in slightly abridged form. However, we additionally need to account for the fact that, unlike in \cite{arendt02}, $\Delta (Uf)$ is initially only defined in the sense of distributions.

\begin{proof}[Proof of Lemma~\ref{lem:intertwining-form}]
Let $y_0 \in \omega_2$. Choose an open neighbourhood $B_1 \subset \omega_1$ of $\tau (y_0)$ and an open neighbourhood $B_2 \subset \omega_2$ of $y_0$ such that $\tau (B_2) \subset B_1$. Now let $f \in C_c^\infty (\omega_1)$ such that $f \equiv 1$ on $B_1$. Then
\begin{displaymath}
	\Delta (h\cdot f \circ \tau) = h \cdot (\Delta f) \circ \tau = 0
\end{displaymath}
on $B_2$. In particular, $\Delta h = 0$ on $B_2$ in the sense of distributions. But this already implies that $h$ is harmonic and in particular an element of $C^\infty (B_2)$.

Next, fix $j \in \{1,2,\ldots, d\}$ and choose $f\in C_c^\infty(\omega_1)$ such that $f(x):=x_j$ on $B_1$ (where we recall that we are writing $x=(x_1,\ldots,x_d) \in \R^d$). Similarly, write $\tau = (\tau_1,\ldots, \tau_d)$. Then
\begin{displaymath}
	\Delta (h \cdot \tau_j) = \Delta (h \cdot f \circ \tau) = h \cdot (\Delta f) \circ \tau = 0
\end{displaymath}
on $B_2$. Thus $h \cdot \tau_j$ is also harmonic and so in $C^\infty (B_2)$. In particular, since $h\neq 0$, also $\tau_j \in C^\infty (B_2)$.

Since this holds for an arbitrary open set $B_1 \subset \omega_1$ and since $\tau (\omega_2) \subset \omega_1$, so that any open set $B_2$ compactly contained in $\omega_2$ can be treated in this fashion, we conclude $h,\tau_j \in C^\infty (\omega_2)$ for all $j=1,\ldots,d$.

Now we may proceed formally as in the proof of \cite[Proposition~2.4]{arendt02}. So fix $f \in C_c^\infty (\omega_1)$. Unpackaging the identity
\begin{displaymath}
	\Delta (h\cdot f \circ \tau) = h \cdot (\Delta f) \circ \tau,
\end{displaymath}
which we now know to hold pointwise (indeed, both sides are $C^\infty$), and using $\Delta h = 0$, we arrive at
\begin{equation}
\label{eq:spass-mit-nabla}
	2\nabla h \cdot \nabla (f \circ \tau) + h\Delta (f\circ \tau) = h(\Delta f)\circ \tau \qquad \text{on } \omega_2,
\end{equation}
for all $f \in C_c^\infty (\omega_1)$. For such functions, since $f \circ \tau \in C^\infty (\omega_2)$, an elementary calculation using the chain rule gives
\begin{displaymath}
	\Delta (f\circ\tau) = \left[\sum_{j,k=1}^d \left( \frac{\partial^2}{\partial x_j\partial x_k} f\right) \circ \tau\right] \nabla \tau_j \cdot \nabla \tau_k
	+ \left[\sum_{k=1}^d \left( \frac{\partial}{\partial x_k} f\right) \circ \tau \right] \Delta \tau_k.
\end{displaymath}
Inserting this into \eqref{eq:spass-mit-nabla} and simplifying,
\begin{displaymath}
	\left[\sum_{j,k=1}^d \left( \frac{\partial^2}{\partial x_j\partial x_k} f\right)\circ \tau\right] \nabla \tau_j \cdot \nabla \tau_k = (\Delta f) \circ \tau
\end{displaymath}
pointwise on $\omega_2$, for all $f \in C_c^\infty (\omega_1)$. Fixing $\omega$ open, arbitrary, compactly contained in $\omega_1$, and choosing $f \in C_c^\infty (\omega_1)$ such that $f(x) = \frac{1}{2}x_j^2$ on $\omega$, we obtain
\begin{displaymath}
	\nabla \tau_j \cdot \nabla \tau_j = 1 \qquad \text{on } \omega_2,\, j=1,\ldots,d,
\end{displaymath}
while the choice of $f(x)=x_jx_k$ on $\omega$ for $j\neq k$ leads to
\begin{displaymath}
	\nabla \tau_j \cdot \nabla \tau_k = 0 \qquad \text{on } \omega_2,\, j\neq k.
\end{displaymath}
These properties together imply that $\tau$ is an isometry on each connected component of $\omega_2$ (a proof of this assertion is given in \cite[Proposition~2.3]{arendt02}). Now choosing $f(x)=x_j$ on $\omega$, from \eqref{eq:spass-mit-nabla} also follows
\begin{displaymath}
	2\nabla h \cdot \nabla \tau_j  + h\Delta \tau_j= 0 \qquad \text{on } \omega_2,
\end{displaymath}
$j=1,\ldots,d$. Since $\tau$ is locally an isometry, $\Delta \tau_j = 0$ for all $j$, so $\nabla h \cdot \nabla \tau_j = 0$ for all $j$. Since the matrix of derivatives $D\tau$ of $\tau$ is an orthogonal matrix, we conclude that $\nabla h=0$, that is, $h$ is constant on each connected component of $\omega_2$.
\end{proof}

At this juncture, we observe that the existence of a disjointness-preserving intertwining operator does not yet force the domains to be congruent.

\begin{example}
\label{ex:not-yet}
(a) Suppose $\Omega_1 = (0,\pi) \subset \R$ and $\Omega_2 = (0,2\pi)$. Define a bounded, linear operator $U: L^2(\Omega_1 ) \to L^2(\Omega_2)$ by
\begin{displaymath}
	Uf(x):= \begin{cases} f(x) \qquad &\text{if } x \in (0,\pi],\\ -f(2\pi-x) \qquad & \text{if } x \in (\pi,2\pi).\end{cases}
\end{displaymath}
Thus $U$ extends functions $f$ on $\Omega_1$ to $\Omega_2$ by odd reflection in $x=\pi$; for example, if $f(x)=\sin x$ on $(0,\pi)$, then $Uf(x)=\sin x$ on $(0,2\pi)$. We see immediately that $U$ is disjointness-preserving. Moreover, if we set $\mathcal{D}_1:= H^2(\Omega_1) \cap H^1_0 (\Omega_1)$, the domain of definition of the Dirichlet Laplacian on $\Omega_1$, then we claim that $U(\mathcal{D}_1) \subset \mathcal{D}_2:= H^2(\Omega_2) \cap H^1_0 (\Omega_2)$. In fact, this is a standard argument using that the domains are one-dimensional: by Sobolev embedding theorems, we have $\mathcal{D}_1 \subset C^1(\overline{\Omega}_1) \cap C_0 (\Omega_1)$, meaning that if $f \in \mathcal{D}_1$, then $Uf$ is, in particular, in $C^1(\overline{\Omega}_2) \cap C_0 (\Omega_2)$. Since it is also piecewise-$H^2$, it is also globally in $H^2 (\Omega_2)$ and takes on the value $0$ at $0$ and $2$. It now follows easily that $U$ satisfies the intertwining property \eqref{eq:intertwining} for the Dirichlet Laplacian; but $\Omega_1$ and $\Omega_2$ are obviously not congruent.

(b) If in (a) we instead define $U$ by
\begin{displaymath}
	Uf(x):= \begin{cases} f(x) \qquad &\text{if } x \in (0,\pi],\\ f(2\pi-x) \qquad & \text{if } x \in (\pi,2\pi),\end{cases}
\end{displaymath}
that is, by even reflection, then we may show that $U$ is a disjointness-preserving operator from $L^2(\Omega_1)$ to $L^2(\Omega_2)$ which now intertwines the respective Neumann Laplacians; in fact it is also positivity preserving: $|Uf|=U|f|$ for all $f \in L^2(\Omega_1)$. The same example works on spaces of continuous functions, i.e., if $U: C(\overline{\Omega}_1) \to C(\overline{\Omega}_2)$, in which case $U$ is even norm-preserving.

(c) Let $\Omega_1 \subset \R^d$ be an arbitrary open set, let $n \in \N \cup \{\infty\}$, and suppose $\omega_1, \ldots, \omega_n \subset \R^d$ are pairwise disjoint copies of $\Omega_1$, i.e., for each $i=1,\ldots,n$ there exists an isometry $\tau_i : \R^d \to \R^d$ such that $\tau (\omega_i) = \Omega_1$. Take $\Omega_2$ to be any open set containing all the $\omega_i$ and define $U:C_0 (\Omega_1) \to C_0 (\Omega_2)$ by
\begin{displaymath}
	Uf (x) := \begin{cases} f \circ \tau_i (x) \qquad &\text{if } x \in \omega_i,\\ 0 \qquad &\text{otherwise}.\end{cases}
\end{displaymath}
Then $U$ is obviously disjointness-preserving, and one may check that $U$ intertwines the Dirichlet Laplacians on $C_0 (\Omega_1)$ and $C_0 (\Omega_2)$.
\end{example}

In all these examples, $U$ is not an isometry, but it acts as a (disjoint) composition of isometries, as the above lemmata already suggest: say, the set $U(\Omega_1):= \{x \in \Omega_1: \exists f \in C_0 (\Omega_1) \text{ with } Uf(x) \neq 0 \}$ is isometric to a finite number of disjoint copies of $\Omega_2$. In the next section, we shall see that \emph{any} disjointness-preserving operator intertwining Dirichlet Laplacians on $C_0$ has this property. In particular, if we make further assumptions on $U$---for example, that $U$ is unitary, but in practice we need much less---then $U$ is in fact an isometry. A formalisation of this observation in different settings, namely the Dirichlet and Neumann Laplacians on spaces of continuous functions or $L^2$, will be the subject of the coming sections.

\section{Disjointness-preserving operators intertwining Dirichlet Laplacians on $C_0$}
\label{sec:dirichlet}

We start with the space $C_0$ (see Definition~\ref{def:c0}). Our first theorem shows that Example~\ref{ex:not-yet}(c) essentially characterises all disjointness-preserving operators intertwining the Dirichlet Laplacians $\DLC$ on $C_0$, up to constants.

\begin{theorem}
\label{thm:disj-pres-dirichlet-char-c0}
Suppose that $\Omega_1,\Omega_2 \subset \R^d$ are open sets, that $\Omega_1$ is connected, and that $0 \neq U: C_0 (\Omega_1) \to C_0 (\Omega_2)$ is a bounded, linear operator such that
\begin{enumerate}
\item[(a)] $f\cdot g= 0$ implies $(Uf)\cdot (Ug) = 0$ for all $f,g \in C_0 (\Omega_1)$; and
\item[(b)] $U(\Delta f) = \Delta (Uf)$ in the sense of distributions, for all $f \in C_c^\infty (\Omega_1)$.
\end{enumerate}
Then there exist pairwise disjoint, connected open sets $\omega_i \subset \Omega_2$, $i \in I\subset \N$, together with isometries $\tau_i: \R^d \to \R^d$ such that $\tau_i (\omega_i)=\Omega_1$ and constants $c_i \in \C$, not all zero, $i \in I $, such that for all $f \in C_0 (\Omega_1)$,
\begin{displaymath}
	Uf(x)= \begin{cases} c_i f \circ \tau_i (x) \qquad &\text{if } x \in \omega_i,\\ 0 \qquad &\text{otherwise}.\end{cases}
\end{displaymath}
\end{theorem}

\begin{remark}
(a) We observe explicitly that $C_c^\infty (\Omega_1)$ is \emph{not} a core for the Dirichlet Laplacian on $C_0 (\Omega_1)$. Thus the assumptions do \emph{not} require that $U$ intertwine the Laplacians on a core; this is another sense in which this is a generalisation of previous results, cf.~\cite[Theorem~2.2]{arendt02} or \cite[Section~3]{arendt01}.

(b) It is clear that the set $I$ is at most countable, and in fact finite whenever $\Omega_2$ is bounded.

(c) The converse of Theorem~\ref{thm:disj-pres-dirichlet-char-c0} is also true. Let $\Omega_1,\Omega_2 \subset \R^d$ be open sets and assume that $\Omega_1$ is connected, and that $\omega_i$ are open sets in $\Omega_2$, $i \in I$, which are pairwise disjoint and isometric to $\Omega_1$. Let $\tau_i$ be isometries such that $\tau_i (\omega_i) = \Omega_1$ for $i \in I$ and let $c_i \in \C$, $i \in I$. Then
\begin{displaymath}
	(Uf)(x):= \begin{cases} c_i f \circ \tau_i (x) \qquad &\text{if } x \in \omega_i,\\ 0 \qquad &\text{otherwise}\end{cases}
\end{displaymath}
defines a disjointness-preserving operator $U: C_0(\Omega_1) \to C_0(\Omega_2)$ which also satisfies the intertwining property (b).
\end{remark}

\begin{proof}[Proof of Theorem~\ref{thm:disj-pres-dirichlet-char-c0}]
Applying Lemma~\ref{lem:disjointness-form} on $\omega_1 = \Omega_1$ and $\omega_2 = \Omega_2$ and then Lemma~\ref{lem:intertwining-form} on $\Omega_1$ and the non-empty open set
\begin{displaymath}
	\Omega_2' := \{ y \in \Omega_2 : \exists f \in C_0(\Omega_1) \text{ with } Uf(y) \neq 0 \} \subset \Omega_2,
\end{displaymath}
we obtain that on each connected component $\omega$ of $\Omega_2'$ there exist a constant $c = c(\omega) \in \C \setminus \{0\}$ and an isometry $\tau:\omega \to \tau(\omega) \subset \Omega_1$ (which extends canonically to an isometry $\tau : \R^d \to \R^d$), such that $(Uf)|_\omega = c f\circ \tau |_\omega$. Then, by continuity, $(Uf)|_{\overline{\omega}} = cf\circ \tau|_{\overline{\omega}}$ for the same constant $c$ and the same isometry $\tau$, for all $f \in \overline{\omega}$ (or equivalently all $f \in C(\overline{\Omega}_1)$). We need to show that in fact $\tau (\omega) = \Omega_1$.

We claim that $\tau (\partial\omega) \subset \partial\Omega_1$. Indeed, suppose $y_0 \in \partial\omega$. Then $\tau (y_0) \in \overline{\Omega}_1$, since $\tau$ is continuous on $\overline{\omega}$ and $\tau (\omega) \subset \Omega_1$. Now suppose for a contradiction that $\tau (y_0) \in \Omega_1$. Choose $f \in C_c^\infty (\Omega_1)$ such that $f(\tau (y_0))=1$. As noted above, continuity of $U$ implies that $Uf(y_0)=c$. But in fact $Uf(y_0)=0$, since either
\begin{enumerate}
\item[(i)] $y_0 \in \partial\omega \cap \partial\Omega_2$, in which case $Uf(y_0)=0$ as $Uf \in C_0(\Omega_2)$, or
\item[(ii)] $y_0 \in \partial\omega \cap \Omega_2 = \partial\Omega_2' \cap \Omega_2$, in which case $y_0 \not\in \Omega_2'$ since the latter is open. Thus $Uf(y_0)=0$ by definition of $\Omega_2'$.
\end{enumerate}
This contradiction proves the claim. To summarise, we have $\tau (\omega) \subset \Omega_1$, and $\partial \tau (\omega) = \tau (\partial\omega) \subset \partial\Omega_1$ (where the equality follows since $\tau$ is an isometry). Since $\Omega_1$ is connected, $\tau (\omega) = \Omega_1$, as required.
\end{proof}

Our next theorem will give additional conditions under which $\Omega_1$ and $\Omega_2$ are isometric. We first need a couple of technical results.

\begin{definition}
\begin{enumerate}
\item[(a)] An open set $\Omega \subset \R^d$ is said to be \emph{regular in topology} if $\interior{\overline{\Omega}} = \Omega$, equivalently, if $B(z,r) \setminus \Omega$ has non-empty interior for all $z \in \partial\Omega$ and all $r>0$.
\item[(b)] An open set $\Omega \subset \R^d$ is \emph{regular in measure} if $|B(z,r)\setminus \Omega |>0$ for all $z \in \partial\Omega$ and all $r>0$.
\end{enumerate}
\end{definition}

\begin{lemma}
\label{lem:regularity}
Let $\Omega,\omega \subset \R^d$ be open sets such that $\omega \subset \Omega$.
\begin{enumerate}
\item[(a)] If $\omega$ is regular in topology and $\Omega\setminus \omega$ has empty interior, then $\omega = \Omega$.
\item[(b)] If $\omega$ is regular in measure and $|\Omega \setminus \omega| = 0$, then $\omega = \Omega$.
\end{enumerate}
\end{lemma}

This is easy to see; we also refer to \cite[Section~3]{arendt02} for more information on this notion. Regularity in topology is (strictly) stronger than regularity in measure; moreover, Lipschitz boundary implies regularity in topology.

We also need the following result, which states that under minimal regularity conditions, if one domain is contained in another and the two share a $k$th Dirichlet Laplacian eigenvalue for some $k\geq 1$, then the two domains are actually equal. This may be considered as a very special case of Kac' problem, which to date seems only to be known for $k=1$ (see \cite{gesztesy94} or \cite{arendt95}).

\begin{theorem}
\label{thm:equal-dirichlet-eigenvalues}
Let $\omega_1,\omega_2 \subset \R^d$ be open sets such that $\omega_1$ is regular in topology, $\omega_1 \subset \omega_2$, and $|\omega_2| < \infty$. If there exists $k\geq 1$ such that $\lambda_k (-\Delta^D_{L^2(\omega_1)}) = \lambda_k (-\Delta^D_{L^2(\omega_2)})$, then $\omega_1 = \omega_2$.
\end{theorem}

For brevity, in what follows we will always write $\lambda_k^D(\Omega)$ for $\lambda_k (-\Delta^D_{L^2(\Omega)})$, for a domain $\Omega \subset \R^d$.

\begin{proof}
Suppose that $\lambda_k^D(\omega_1) =  \lambda_k^D(\omega_2)$. By the standard Courant--Fischer minimax formula,
\begin{equation}
\label{eq:dirichlet-minimax}
	\lambda_k^D (\omega_i) = \min_{\substack{X \subset H^1_0 (\omega_i)\\ \dim X = k}} \max_{\substack{u \in X\\ \|u\|_{L^2(\omega_i)}=1}}
	\int_{\omega_i} |\nabla u|^2\,\textrm{d}x,
\end{equation}
for $i=1,2$. Moreover, if $X$ is a $k$-dimensional subspace of $H^1_0 (\omega_i)$ realising the minimum in \eqref{eq:dirichlet-minimax}, then $X$ contains an eigenfunction corresponding to $\lambda_k^D (\omega_i)$ (see \cite[Lemma~4.1(1)]{berkolaiko19}).

Now let $X \subset H^1_0 (\omega_1)$ be a minimising subspace for $\lambda_k^D (\omega_1)$, $\dim X = k$. For every $u \in X$ we set $\tilde u \in H^1_0 (\omega_2)$ to be the function $u$ extended by $0$ on $\omega_2 \setminus \omega_1$. Then $\widetilde X := \{ \tilde u : u \in H^1_0 (\omega_1) \}$ is a $k$-dimensional subspace of $H^1_0 (\omega_2)$ ; moreover, it follows from the assumption that $\lambda_k^D (\omega_1) = \lambda_k^D (\omega_2)$ that $\widetilde X$ realises the minimum in \eqref{eq:dirichlet-minimax} for $i=2$. Hence there exists an eigenfunction $\tilde\psi \in X$ of the operator $-\Delta^D_{L^2(\omega_2)}$. In particular, $\tilde\psi$ is real analytic on $\omega_2$.

On the other hand, by construction $\tilde\psi|_{\omega_2 \setminus \omega_1} = 0$. This implies that $\interior (\overline{\omega_2 \setminus \omega_1}) = \emptyset$. Thus $\omega_2 \setminus \overline{\omega}_1 = \emptyset$, and so $\omega_2 \subset \overline{\omega}_1$, whence $\omega_2 \subset \interior (\overline{\omega}_1)$. Using the topological regularity of $\omega_1$, we conclude that $\omega_1 = \omega_2$.
\end{proof}

We are now in a position to state our second main theorem, which gives conditions under which the domains of Theorem~\ref{thm:disj-pres-dirichlet-char-c0} are indeed congruent.

\begin{theorem}
\label{thm:disj-pres-dirichlet-c0-isom}
Adopt the assumptions of Theorem~\ref{thm:disj-pres-dirichlet-char-c0} and assume that one of the following further conditions is satisfied:
\begin{enumerate}
\item[(a)] $\Omega_1$ is regular in measure and $|\Omega_1| = |\Omega_2| < \infty$; or
\item[(b)] $\Omega_1$ is regular in topology, $|\Omega_1|, |\Omega_2| < \infty$, and there exists $k \geq 1$ such that $\lambda_k^D (\Omega_1) = \lambda_k^D (\Omega_2)$, i.e., the two $L^2$-Dirichlet Laplacians share an eigenvalue; or
\item[(c)] $U:C_0(\Omega_1) \to C_0 (\Omega_2)$ has dense range; or
\item[(d)] $\Omega_2$ is regular in measure, $|\Omega_2|<\infty$ and there exists another operator $\widetilde U: C_0 (\Omega_2) \to C_0 (\Omega_1)$ satisfying the assumptions of Theorem~\ref{thm:disj-pres-dirichlet-char-c0}, with the roles of $\Omega_1$ and $\Omega_2$ interchanged.
\end{enumerate}
Then there exist an isometry $\tau: \R^d \to \R^d$ with $\tau (\Omega_2) = \Omega_1$ and a constant $c \in \C \setminus \{0\}$ such that
\begin{displaymath}
	Uf = cf \circ \tau \qquad \text{for all } f \in C_0 (\Omega_1).
\end{displaymath}
In particular, $\Omega_1$ and $\Omega_2$ are congruent.
\end{theorem}

\begin{proof}
We adopt the notation from the proof of Theorem~\ref{thm:disj-pres-dirichlet-char-c0}.

(a) If $|\Omega_1|=|\Omega_2|$, then $I=\{i_0\}$ is a singleton. Since $\tau_{i_0} (\omega_{i_0}) = \Omega_1$, also $\omega_{i_0} \subset \Omega_2$ is regular in measure. But $|\omega_{i_0}| = |\Omega_1| = |\Omega_2|$ by hypothesis; thus $|\Omega_2 \setminus \omega_{i_0}|=0$. It follows from Lemma~\ref{lem:regularity} that $\omega_{i_0} = \Omega_2$. Now the claim follows.

(b) Note that since $|\Omega_1|,|\Omega_2| < \infty$, the respective $L^2$-Dirichlet spectra are discrete. Now consider any $i \in I$. Since $\tau_i (\omega_i) = \Omega_1$ for all $i$, we have $\lambda_k^D (\Omega_2) = \lambda_k^D (\Omega_1) = \lambda_k^D (\omega_i)$. Since $\Omega_1$ is regular in topology, so is $\omega_i$. It now follows from Theorem~\ref{thm:equal-dirichlet-eigenvalues} that $\omega_i = \Omega_2$, which implies the claim.

(c) Assume that $I$ has at least two elements, say $i_1,i_2 \in I$, $i_1 \neq i_2$. Since $\tau_{i_1} (\omega_{i_1}) = \tau_{i_2} (\omega_{i_2}) = \Omega_1$, we can find $y_1 \in \omega_1$, $y_2 \in \omega_2$ and $z \in \Omega_1$ such that $\tau_{i_1} (y_1) = \tau_{i_2} (y_2) = z$. Then $(Uf)(y_1) = c_1 f(z)$, while $(Uf)(y_2) = c_2 f(z)$ for all $f \in C_0 (\Omega_1)$. Choose $\alpha,\beta \in \R$, not both zero, such that $\alpha c_1 + \beta c_2 = 0$. Then $\alpha g(y_1) + \beta g(y_2) = 0$ for all $g$ in the range of $U$, meaning that this range is not dense.

Thus the assumption (c) implies that $I$ is a singleton $\{i_0\}$. Assume now that $\Omega_2 \setminus \omega_{i_0} \neq \emptyset$ and choose $y_0 \in \Omega_2 \setminus \omega_{i_0}$. Then $(Uf)(y_0)=0$ for all $f \in C_0 (\Omega_1)$, which implies that $U$ does not have dense range. We conclude that $\Omega_2 = \omega_{i_0}$.

(d) We know already that there exist a subset $\omega$ of $\Omega_2$ and an isometry $\tau$ such that $\Omega_1 = \tau^{-1}(\omega)$; in particular, $|\Omega_1|\leq |\Omega_2| < \infty$. By assumption, there now exist $\tilde\omega \subset \Omega_1$ and an isometry $\tilde\tau$ such that $\Omega_2 = \tilde\tau^{-1}(\tilde\omega)$. The only possibility is that $|\Omega_1| = |\Omega_2|$, and the claim follows from (a).
\end{proof}

Finally, we return to the types of intertwining operators corresponding to Kac' problem (see Proposition~\ref{prop:unitary} and the discussion around it). If we combine Theorem~\ref{thm:disj-pres-dirichlet-c0-isom} with the additional assumption that the intertwining operator in question is invertible, then we obtain a positive result.

\begin{corollary}
\label{cor:dirichlet-c0-kac-disj}
Suppose $\Omega_1 \subset \R^{d_1}$ and $\Omega_2 \subset \R^{d_2}$, $d_1,d_2\geq 1$, are two bounded open sets which are Dirichlet regular (cf.~Remark~\ref{rem:dirichlet-regular}), such that $\Omega_1$ is connected. Suppose also that $U: C_0(\Omega_1) \to C_0(\Omega_2)$ is an invertible disjointness-preserving operator which intertwines Dirichlet Laplacians on $C_0 (\Omega_1)$ and $C_0(\Omega_2)$ in the sense of Definition~\ref{def:intertwining}. Then $d_1=d_2=:d$ and there exist an isometry $\tau: \R^d \to \R^d$ with $\tau (\Omega_2) = \Omega_1$ and a constant $c \in \C \setminus \{0\}$ such that
\begin{displaymath}
	Uf = cf \circ \tau \qquad \text{for all } f \in C_0 (\Omega_1).
\end{displaymath}
\end{corollary}

\begin{proof}
By assumption, condition (3) of Proposition~\ref{prop:unitary} holds for the Dirichlet Laplacians $\Delta^D_{C_0(\Omega_1)}$ and $\Delta^D_{C_0(\Omega_2)}$ on $C_0$. Since $\Omega_1$ and $\Omega_2$ are Dirichlet regular, Proposition~\ref{prop:unitary} yields that the spectra of $\Delta^D_{C_0(\Omega_1)}$ and $\Delta^D_{C_0(\Omega_2)}$ coincide. But Dirichlet regularity of $\Omega_1$ and $\Omega_2$ also guarantees that in each case the Dirichlet Laplacian has the same spectrum on $C_0$ as it does on $L^2$ (as already noted in the proof of that proposition; again, see \cite[Theorem~2.3]{arendt99}). We may thus conclude that the spectra of $\Delta^D_{L^2(\Omega_1)}$ and $\Delta^D_{L^2(\Omega_2)}$ coincide. We next recall Weyl's law in the form
\begin{equation}
\label{eq:weyl}
	\lim_{\lambda \to \infty} \frac{N_i(\lambda)}{\lambda^{d_i/2}} = \frac{\omega_{d_i} |\Omega_i|}{(4\pi)^{d_i/2}},
\end{equation}
$i=1,2$, where $N_i(\lambda) = \# \{n\geq 1: \lambda_n^D(\Omega_i) \leq \lambda \}$ is the eigenvalue counting function associated with $-\Delta^D_{L^2(\Omega_i)}$ and $\omega_{d_i}$ is the volume of the ball of unit radius in $\R^{d_i}$. (See \cite[Section~1.6]{arendt09} or \cite[Theorem~1.11]{birman80} for a proof of \eqref{eq:weyl} valid under our regularity assumptions.) Since $N_1 (\lambda) = N_2 (\lambda)$, formula \eqref{eq:weyl} implies that $d_1=d_2=:d$, since the limit can be finite and non-zero for at most one choice of $d_i$. We may now invoke Theorem~\ref{thm:disj-pres-dirichlet-c0-isom}(c) to complete the proof.
\end{proof}

\section{Disjointness-preserving operators intertwining Neumann and Robin Laplacians on spaces of continuous functions}
\label{sec:neumann}

We now wish to perform an analysis similar to the one of Section~\ref{sec:dirichlet}, but where our operator $U$ intertwines Neumann or Robin Laplacians in a suitable weak sense. To keep things as non-technical as possible, in this section we work with domains satisfying a modest regularity condition: we will assume unless explicitly stated otherwise that $\Omega_1,\Omega_2 \subset \R^d$, $d\geq 2$ are bounded, connected open sets with Lipschitz boundary, or bounded open intervals in dimension $d=1$. To keep the notation simpler, in this section we will also write $\Delta_\Omega^N$ for $\NLC$ and $\Delta_\Omega^\beta$ for $\Delta_{C(\overline{\Omega})}^\beta$, as we will be working only on~$C$.

We suppose throughout that $U: C(\overline{\Omega}_1) \to C(\overline{\Omega}_2)$ is a bounded linear operator satisfying the following disjointness preservation and intertwining assumptions:
\begin{enumerate}
\item[(a)] $f \cdot g = 0$ implies $(Uf) \cdot (Ug) = 0$ for all $f,g \in C(\overline{\Omega}_1)$; and
\item[(b)] $U (C_c^\infty (\Omega_1)) \subset D(\Delta^{\beta}_{\Omega_2})$ for some $\beta \in C(\partial\Omega_2)$, and $U(\Delta f) = \Delta (Uf)$ for all $f \in C_c^\infty (\Omega_1)$.
\end{enumerate}
We emphasise that $\Delta^{\beta}_{\Omega_2}$ reduces to the Neumann Laplacian $\Delta^N_{\Omega_2}$ if $\beta \equiv 0$; and indeed if $U$ intertwines the Neumann Laplacians in the sense of Definition~\ref{def:intertwining}, then it satisfies (b) with $\beta \equiv 0$. But in fact, as in Section~\ref{sec:dirichlet}, this condition is considerably weaker, since $C_c^\infty (\Omega_1)$ is certainly not a core for the Neumann Laplacian (or any Robin Laplacian) on $C(\overline{\Omega}_1)$. For example, if $U$ intertwines any two Robin Laplacians, $\Delta_{\Omega_1}^{\beta_1}$ on $\Omega_1$ and $\Delta_{\Omega_2}^{\beta_2}$ on $\Omega_2$, where $\beta_1 \in C (\partial\Omega_1)$, $\beta_2 \in C (\partial\Omega_2)$, then it satisfies (b) with $\beta = \beta_2$.

To (a) and (b) we add the weak non-degeneracy assumption
\begin{enumerate}
\item[(c)] there exists $f \in C_c^\infty(\Omega_1)$ such that $Uf \neq 0$.
\end{enumerate}
As in Section~\ref{sec:dirichlet}, we set
\begin{displaymath}
	\Omega_2' := \{ y \in \Omega_2 : \exists f \in C(\overline{\Omega}_1) \text{ with } Uf(y) \neq 0 \} \subset \Omega_2,
\end{displaymath}
which is open, and non-empty by assumption (c). We also set $h :=U1 \in C(\overline{\Omega}_2)$. Our first result is a direct application of Lemma~\ref{lem:disjointness-form} and Lemma~\ref{lem:intertwining-form}.

\begin{lemma}
\label{lem:neumann-form-1}
Under the above assumptions (a), (b) and (c), there exists a unique continuous mapping $\tau: \Omega_2' \to \overline{\Omega}_1$ such that for all $f \in C(\overline{\Omega}_1)$,
\begin{equation}
\label{eq:neumann-form-1}
	Uf (y) = \begin{cases} h(y) f(\tau(y)) \qquad &\text{if } y \in \Omega_2',\\ 0 &\text{if } y \in \Omega_2 \setminus \Omega_2',\end{cases}
\end{equation}
and $\Omega_2' = \{y \in \Omega_2: h(y)=0\}$. Moreover, $\Omega_2' \cap \tau^{-1} (\Omega_1)$ is nonempty and open, and if $\omega$ is any connected component of $\Omega_2' \cap \tau^{-1} (\Omega_1)$, then there exists an isometry $\bar{\tau} : \R^d \to \R^d$ and a constant $c \in\C \setminus \{0\}$ such that $\tau|_\omega = \bar{\tau}|_\omega$ and $h|_\omega = c$.
\end{lemma}

\begin{proof}
We start by applying Lemma~\ref{lem:disjointness-form} in the form of Remark~\ref{rem:disjointness-form-c} to obtain \eqref{eq:neumann-form-1} together with the continuity of $\tau$; that the $h$ of Lemma~\ref{lem:disjointness-form} equals $U1$ simply follows from setting $f=1$ in \eqref{eq:neumann-form-1}. Now it is possible that $\tau (\Omega_2') \cap \partial\Omega_1 \neq \emptyset$; we have to show that $\tau (\Omega_2') \not\subset \partial\Omega_1$. Indeed, suppose the opposite. Then, for any $f \in C_c^\infty (\Omega_1)$, since $f(\partial\Omega_1) = 0$, we would have $Uf=0$. But this contradicts assumption (c). Thus $\Omega_2' \cap \tau^{-1} (\Omega_1)$ has a non-empty, open connected component $\omega$ as claimed. We may now apply Lemma~\ref{lem:intertwining-form} to it to obtain the other conclusions.
\end{proof}

If we now make more careful use of the condition on the domains contained in assumption (b), we can show that $\omega$ is actually isometric to $\Omega_1$. In doing so, we also make direct use of the assumption that $\Omega_2$ has Lipschitz boundary, although we expect that with more effort this assumption could be weakened. Actually, the following lemma is the only place in this section (apart from Corollary~\ref{cor:neumann-c0-kac-disj}) where the assumption (b), that $UC_c^\infty(\Omega_1) \subset D(\Delta^N_{\Omega_2})$ or $D(\Delta^{\beta}_{\Omega_2})$, enters explicitly.

\begin{lemma}
\label{lem:neumann-form-2}
With the assumptions and notation of Lemma~\ref{lem:neumann-form-1}, $\tau(\omega) = \Omega_1$.
\end{lemma}

\begin{proof}
1. Suppose $\tau (\omega) \neq \Omega_1$. Since by definition $\omega \subset \tau^{-1}(\Omega_1)$, i.e., $\tau(\omega) \subset \Omega_1$, this can only be the case if there exists some $z \in \partial\omega$ such that $\bar{\tau} (z) \in \Omega_1$.

2. We claim that $z \in \partial\Omega_2$. To see this, take $\omega \ni z_n \to z$; then
\begin{displaymath}
	Uf(z) = \lim_{n\to\infty} cf(\tau (z_n)) = cf(\bar{\tau}(z))
\end{displaymath}
for all $f \in C_c^\infty(\Omega_1)$. If $z \in \Omega_2$, then this implies $z \in \Omega_2'$ (and $\tau(z)=\bar{\tau}(z)$). Since also $z \in \tau^{-1}(\Omega_1)$ by assumption and $\omega$ is open and closed in $\Omega_2' \cap \tau^{-1}(\Omega_1)$, we conclude $z \in \omega$. But this contradicts the fact that $\omega$ is open (in $\R^d$), and thus $z \in \partial\omega$ as claimed.

3. Since $\bar{\tau} (z) \in \Omega_1$, there exists $r>0$ small such that $\bar{\tau} (B(z,r)) \subset \Omega_1$. Set $B_1 := B(z,r)$ and observe that $\bar{\tau} (B_1) = B(\bar{\tau}(z), r) =:B_2$. Also note that $f \mapsto cf\circ \bar{\tau}$ defines a bijective map from $C_c^\infty (B_1)$ to $C_c^\infty (B_2)$. But $Uf = cf\circ \bar{\tau} \in D(\Delta^{\beta}_{\Omega_2})$ for all $f \in C_c^\infty (B_1)$. This means that $g \in D(\Delta^{\beta}_{\Omega_2})$ for all $g \in C_c^\infty (B_2)$.

But by Step 2, $\partial\Omega_2 \cap B_2 \neq 0$, and in fact this intersection must have positive surface measure (since $\partial\Omega_2$ is closed and Lipschitz). It now follows that $\nabla g \cdot \nu = -\beta g$ $\sigma$-a.e. on $\partial\Omega_2 \cap B_2$ for all $g \in C_c^\infty (B_2)$, where $\nu$ is the outer unit normal to $\Omega_2$ and $\sigma$ the surface measure on $\partial\Omega_2$. This is a contradiction; hence the assumption that $\tau (\omega) \neq \Omega_1$ is false.
\end{proof}

Thus the assumptions (a), (b) and (c) together imply that there exists a copy of $\Omega_1$ in $\Omega_2$ such that $\tau$ maps $\Omega_1$ isometrically onto this copy in $\Omega_2$. However, these hypotheses by themselves are too weak to allow us to say much more, as the following example shows.

\begin{example}
\label{ex:neumann-ex-1}
Let $\Omega_1 := \{x \in \R^2: |x|<1\}$ and $\Omega_2 := \{x \in \R^2: |x|<2\}$, and set
\begin{displaymath}
	\tau (x):=\begin{cases} x \qquad &\text{if } |x|<1,\\ \frac{x}{|x|} \qquad &\text{if } |x| \geq 1.\end{cases}
\end{displaymath}
Now let $h : \overline{\Omega}_2 \to \R$ be any continuous function such that $h|_{\Omega_1} \equiv 1$. If we define $U:C(\overline{\Omega}_1) \to C(\overline{\Omega}_2)$ by $Uf(y) := h(y)f(\tau (y))$, then $U$ acts on $C_c^\infty (\Omega_1)$ by extension by zero to $\Omega_2 \setminus \Omega_1$ and thus satisfies conditions (b) (for any valid $\beta$) and (c). One may also check that it satisfies (a), since $\tau$ preserves disjoint supports. Here, we have $\omega = \Omega_1$, and $\tau$ is the identity on $\omega$. The set $\Omega_2'$ depends on the particular choice of $h$ and may thus be any open set which compactly contains $\Omega_1$.
\end{example}

In the sequel, we will thus strengthen assumption (c) to the following non-degeneracy condition:
\begin{enumerate}
\item[(d)] Let $\omega \subset \Omega_2$ be open. If $(Uf)|_\omega = 0$ for all $f \in C_c^\infty (\Omega_1)$, then $\omega=\emptyset$.
\end{enumerate}
This holds if, for example, for every $y \in \Omega_2$ there exists some $f \in C_c^\infty (\Omega_1)$ with $Uf(y) \neq 0$, or if $U$ is an invertible intertwining operator.

\begin{lemma}
\label{lem:neumann-form-d}
Under assumptions (a), (b) and (d), $\Omega_2' \cap \tau^{-1} (\Omega_1)$ is dense in $\Omega_2$.
\end{lemma}

\begin{proof}
Let $\omega \subset \Omega_2$ be an open set such that $\omega \cap \Omega_2' \cap \tau^{-1} (\Omega_1) = \emptyset$. Then, writing $Uf = h f \circ \tau$ as in \eqref{eq:neumann-form-1}, for $y \in \omega$ we have either $y \not\in \Omega_2'$ and thus $h(y)=0$, or $y \in \Omega_2' \setminus \tau^{-1}(\Omega_1)$ and thus $\tau(y) \in \partial\Omega_1$. In either case, if $f \in C_c^\infty(\Omega_1)$ is arbitrary, then we have $Uf(y)=0$. This contradicts (d).
\end{proof}

We can now give a characterisation of $U$ under these assumptions, which we repeat for ease of reference.

\begin{theorem}
\label{thm:disj-pres-neumann-char-c}
Assume $\Omega_1$ and $\Omega_2$ are bounded, connected Lipschitz domains in $\R^d$, $d\geq 2$, or bounded open intervals if $d=1$. Let $U: C(\overline{\Omega}_1) \to C(\overline{\Omega}_2)$ satisfy the assumptions
\begin{enumerate}
\item[(a)] $f \cdot g = 0$ implies $(Uf) \cdot (Ug) = 0$ for all $f,g \in C(\overline{\Omega})$; and
\item[(b)] $U(C_c^\infty (\Omega_1)) \subset D(\Delta^{\beta}_{\Omega_2})$ for some $\beta \in C(\partial\Omega_2)$, and $U(\Delta f)=\Delta (Uf)$ for all $f \in C_c^\infty(\Omega_1)$; and
\item[(c)] if $\omega \subset \Omega_2$ is open and $Uf|_\omega = 0$ for all $f \in C_c^\infty (\Omega_1)$, then $\omega=\emptyset$.
\end{enumerate}
Then there exist a constant $c \in \C \setminus \{0\}$, pairwise disjoint, connected open sets $\omega_1,\ldots,\omega_n \subset \Omega_2$ and isometries $\bar{\tau}_1,\ldots,\bar{\tau}_n : \R^d \to \R^d$, $n \geq 1$, such that $\bar{\tau}_i (\omega_i) = \Omega_1$ for all $i=1,\ldots,n$, $\bigcup_{i=1}^n \omega_i$ is dense in $\Omega_2$, and
\begin{displaymath}
	Uf(y) = cf(\bar{\tau}_i (y)) \qquad \text{for all $y \in \overline{\omega}_i$, for all } i=1,\ldots,n.
\end{displaymath}
\end{theorem}

We emphasise that the constant $c$ does not depend on $i=1,\ldots,n$.

\begin{proof}
The set $\Omega_2' \cap \tau^{-1} (\Omega_1)$ is open and, by Lemma~\ref{lem:neumann-form-d}, dense in $\Omega_2$. Now Lemma~\ref{lem:neumann-form-2} can be applied to each connected component $\omega$ of $\Omega_2' \cap \tau^{-1} (\Omega_1)$. In particular, since $\Omega_2$ is bounded and each component is congruent to $\Omega_1$, there can be only finitely many, call them $\omega_1,\ldots,\omega_n$. Together with Lemma~\ref{lem:neumann-form-1}, this also yields isometries $\bar{\tau}_i$ such that $\bar{\tau}_i (\omega_i) = \Omega_1$ and constants $c_i \neq 0$, $i=1,\ldots,n$, such that $Uf(y) = c_i f(\bar{\tau}_i (y))$ for all $y \in \omega_i$.

Since $\Omega_2' \cap \tau^{-1} (\Omega_1)$ is the disjoint union of $\omega_1,\ldots,\omega_n$ and is dense in $\Omega_2$, it follows that $h (\overline{\Omega}_2) \subset \{ c_1,\ldots,c_n\}$, where $h$ is as in Lemma~\ref{lem:neumann-form-1}. In particular, $\Omega_2' = \Omega_2$. Since $\Omega_2$ is connected and $h^{-1} (\{c_i\})$ is open and closed in $\Omega_2$ for each $i=1,\ldots,n$, it follows that $c_1 = \ldots = c_n =:c$.
\end{proof}

Example~\ref{ex:not-yet}(b) shows that $n\geq 2$ is possible in Theorem~\ref{thm:disj-pres-neumann-char-c}; in fact, by modifying the example via repeated even reflection to $\Omega_1 = (0,\pi)$, $\Omega_2 = (0,n\pi)$, i.e., so that $U$ has the form
\begin{displaymath}
	Uf(x) := \begin{cases} f(x-2k\pi) \qquad &\text{if } x \in (2k\pi, (2k+1)\pi],\\ f((2k+2)\pi-x) \qquad &\text{if } x \in ((2k+1)\pi,(2k+2)\pi),\end{cases}
\end{displaymath}
$k=0,\ldots,n/2-1$, we can construct an example for any $n\geq 1$ even, with an easy variant for $n$ odd. However, under certain additional assumptions on $\Omega_1$ and $\Omega_2$, or $U$, we can conclude that $n=1$.

\begin{theorem}
\label{thm:disj-pres-neumann-c-isom}
Suppose, in addition to the assumptions of Theorem~\ref{thm:disj-pres-neumann-char-c}, that one of the following conditions holds:
\begin{enumerate}
\item[(d)] $|\Omega_1|=|\Omega_2|$; or
\item[(e)] $U: C(\overline{\Omega}_1) \to C(\overline{\Omega}_2)$ has dense range.
\end{enumerate}
Then there exist an isometry $\tau: \R^d \to \R^d$ with $\tau(\Omega_2)=\Omega_1$ and a constant $c \in \C \setminus \{0\}$ such that
\begin{displaymath}
	Uf = cf\circ \tau \qquad \text{for all } f \in C(\overline{\Omega}_1).
\end{displaymath}
In particular, $\Omega_1$ and $\Omega_2$ are congruent.
\end{theorem}

\begin{remark}
In contrast to Theorem~\ref{thm:disj-pres-dirichlet-c0-isom}(b), to conclude that $n=1$ it is not enough that, say, the Neumann Laplacians on $L^2(\Omega_1)$ and $L^2(\Omega_2)$ share an eigenvalue: $0$ is always an eigenvalue, with eigenfunction $1$, on any bounded, connected open set.
\end{remark}

\begin{proof}[Proof of Theorem~\ref{thm:disj-pres-neumann-c-isom}]
We take the assumptions and setup of Theorem~\ref{thm:disj-pres-neumann-char-c}: we merely have to show that $n=1$. In case (d) this is obvious.

So assume (e), and assume that $n\geq 2$. Let $y_1 \in \omega_1$ be arbitrary; then $y_2 := \bar{\tau}_2^{-1} \circ \bar{\tau}_1 (y_1) \in \omega_2$, and so $y_1 \neq y_2$. However, for all $g \in U( C(\overline{\Omega}_1))$ we have $g(y_1)=g(y_2)$; to see this, let $f \in C(\overline{\Omega}_1)$ be such that $g = Uf$. Then $g(y_1) = cf(\bar{\tau}_1(y_1)) = cf(\bar{\tau}_2(y_2)) = g(y_2)$. Thus the range of $U$ is not dense.
\end{proof}

We next give another result where the main assumption concerns the regularity of $\Omega_1$.

\begin{theorem}
\label{thm:disj-pres-neumann-c-c1-isom}
Suppose in addition to the assumptions of Theorem~\ref{thm:disj-pres-neumann-char-c} that $d \geq 2$ and $\Omega_1$ is of class $C^1$. Then there exist an isometry $\tau: \R^d \to \R^d$ with $\tau(\Omega_2)=\Omega_1$ and a constant $c \in \C \setminus \{0\}$ such that
\begin{displaymath}
	Uf = cf\circ \tau \qquad \text{for all } f \in C(\overline{\Omega}_1).
\end{displaymath}
In particular, $\Omega_1$ and $\Omega_2$ are congruent.
\end{theorem}

We recall that all the known counterexamples to Kac' original conjecture are Lipschitz but not $C^1$, being formed by repeated reflection of a given building block, cf.~the introduction; and it seems reasonable to expect Kac' conjecture to hold for $C^1$ domains, cf.~the discussion in \cite{arendt14}. Although the setting here is somewhat different, in Theorem~\ref{thm:disj-pres-neumann-char-c} we can, in the same way, have multiple identical copies of a domain glued together; Theorem~\ref{thm:disj-pres-neumann-c-c1-isom} shows that this is indeed only possible if corners and thus reflections are allowed.

The proof of Theorem~\ref{thm:disj-pres-neumann-c-c1-isom} is a direct consequence of the following statement.

\begin{proposition}
\label{prop:c1-union}
Let $\omega_1,\ldots,\omega_n$ be pairwise disjoint bounded isometric open sets in $\R^d$, $d\geq 2$, of class $C^1$, and let $\Omega \subset \R^d$ be open. Assume that
\begin{displaymath}
	\overline\Omega = \bigcup_{i=1}^n\overline{\omega}_i.
\end{displaymath}
Then $\Omega$ is not a Lipschitz domain.
\end{proposition}

For simplicity, we will give the proof in the special case that $n=2$. We first need the following lemma.

\begin{lemma}
\label{lem:disj-pres-neumann-top}
With the assumptions and notation of Proposition~\ref{prop:c1-union} (and assuming that $n=2$) we have
\begin{enumerate}
\item $\partial \omega_1 \cap \Omega_2 = \partial \omega_2 \cap \Omega$;
\item $\partial \omega_1 \cap \partial \omega_2 \neq \emptyset$;
\item there exists a connected component $\Gamma$ of $\partial\Omega$ such that $\partial \omega_1 \cap \Gamma \neq \emptyset$ and $\partial \omega_2 \cap \Gamma \neq \emptyset$;
\item there exists a point $z \in \partial\omega_1 \cap \partial\omega_2 \cap \Gamma$, where $\Gamma$ is the connected component of $\partial\Omega$ from (3).
\end{enumerate}
\end{lemma}

If $n>2$, then (1) may be modified to read $z \in \Omega \cap \bigcup_{i=1}^n \partial\omega_i$ if and only if $z \in \partial\omega_i$ for \emph{at least two} $i$ (which may depend on $z$); and there still \emph{exist} $\omega_1$ and $\omega_2$ such that (2), (3) and (4) all hold for this pair.

\begin{proof}
(1) Suppose for a contradiction that $z \in (\partial\omega_1 \cap \Omega) \setminus \partial\omega_2$. Since $\omega_1$ and $\omega_2$ are open and disjoint, also $\partial\omega_1 \cap \omega_2 = \emptyset$ and thus $z \in (\partial\omega_1 \cap \Omega) \setminus \overline{\omega}_2$. Choose $\varepsilon>0$ such that $B(z,\varepsilon) \subset \Omega \setminus \overline{\omega}_2$ and set $V:=B(z,\varepsilon) \cap (\Omega \setminus \overline{\omega}_1)$; then $V \neq \emptyset$ since $z \in \partial\omega_1$. Thus $V$ is an open and non-empty subset of $\Omega$, and $V \cap (\overline{\omega}_1 \cup \overline{\omega}_2) = \emptyset$, contradicting $\Omega \subset \overline{\omega}_1 \cup \overline{\omega}_2$.

(2) By (1), it suffices to prove that $\partial\omega_1 \cap \Omega \neq \emptyset$. But if $\partial\omega_1 \cap \Omega = \emptyset$ and $\omega_1 \subset \Omega$, then $\omega_1 = \Omega$ since the latter is connected, a contradiction to $\emptyset \neq \omega_2 \subset \Omega$.

(3) We let $\Gamma$ be the outermost component of $\partial\Omega$, that is, denoting by $\diam V = \sup\{ \dist(x,y) : x,y\in V\}$ the diameter of an arbitrary set $V \subset \R^d$, we set $\Gamma$ to be the connected component of $\partial\Omega$ for which $\diam \Gamma = \diam \Omega = \diam \partial\Omega$. Now since $\omega_1 \cup \omega_2$ is dense in $\Omega$ by Theorem~\ref{thm:disj-pres-neumann-char-c} (and the assumption that $n=2$), we have $\overline{\omega}_1 \cup \overline{\omega}_2 = \overline{\Omega}$ and hence at least one of $\partial\omega_1$ and $\partial\omega_2$ has non-empty intersection with $\Gamma$; say $\partial\omega_2 \cap \Gamma \neq \emptyset$.

Suppose now that $\partial\omega_1 \cap \Gamma = \emptyset$. This means, firstly, that $\Gamma \subset \partial\omega_2$, and so $\diam \omega_2 = \diam \Omega$. But it also forces $\overline{\omega}_1$ to be compactly contained in the unique bounded open set whose boundary is $\Gamma$, and hence $\diam \omega_1 < \diam \Omega$.
Hence $\omega_1$ and $\omega_2$ cannot be congruent, a contradiction since they are isometric.

(4) The set $\Gamma$ is closed and connected and, since $\overline{\omega}_1 \cup \overline{\omega}_2 = \overline{\Omega}$, we have in particular that $\Gamma \subset \partial\omega_1 \cup \partial\omega_2$. Since $\partial\omega_1$ and $\partial\omega_2$ are also closed sets, we must have $\partial\omega_1 \cap \partial\omega_2 \cap \Gamma \neq \emptyset$.
\end{proof}

\begin{proof}[Proof of Proposition~\ref{prop:c1-union} and hence of Theorem~\ref{thm:disj-pres-neumann-c-c1-isom}]
Let $z \in \partial\omega_1 \cap \partial\omega_2 \cap \partial\Omega$. Then $\partial\omega_1$ and $\partial\omega_2$ are locally $C^1$ manifolds passing through $z$ which cannot cross transversally (since otherwise $\omega_1$ and $\omega_2$ would have non-empty intersection); more precisely, fixing $\varepsilon>0$ sufficiently small and setting $M_j:=B(z,\varepsilon) \cap \partial\omega_j$, $j=1,2$, which are $C^1$-manifolds inside $B(z,\varepsilon)$, for the tangent spaces $T_z M_j$ at $z$ we have $T_z M_1 = T_z M_2$. We claim that $\partial\Omega$ cannot be Lipschitz at $z$: indeed, $\partial\Omega \cap B(z,\varepsilon) = \partial (\R^d \setminus \Omega_2) \cap B(z,\varepsilon)$, and $(\R^d \setminus \Omega_2) \cap B(z,\varepsilon)$ has a singularity (cusp) at $z$, being contained in $B(z,\varepsilon) \setminus (\omega_1 \cup \omega_2)$ and having $z \in \partial\Omega$ as a boundary point. Thus $\partial\Omega$ cannot be Lipschitz at $z$, as it does not satisfy the outer cone condition there.
\end{proof}

We conclude this section by giving a positive result for the Neumann and Robin Laplacians on $C(\overline{\Omega})$ under the assumption that there is a disjointness-preserving invertible intertwining operator, analogous to Corollary~\ref{cor:dirichlet-c0-kac-disj} (cf.~also Proposition~\ref{prop:unitary} and the discussion around it).

\begin{corollary}
\label{cor:neumann-c0-kac-disj}
Suppose $\Omega_1 \subset \R^{d_1}$ and $\Omega_2 \subset \R^{d_2}$ are two bounded, open sets with Lipschitz boundary, such that $\Omega_1$ is connected, and $\beta_1 \in C(\partial\Omega_1)$, $\beta_2 \in C(\partial\Omega_2)$. Suppose also that $U:C(\overline{\Omega}_1) \to C(\overline{\Omega}_2)$ is an invertible disjointness-preserving operator intertwining the Robin Laplacians $-\Delta_{\Omega_1}^{\beta_1}$ on $C(\overline{\Omega}_1)$ and $-\Delta_{\Omega_2}^{\beta_2}$ on $C(\overline{\Omega}_2)$ in the sense of Definition~\ref{def:intertwining}. Then $d_1 = d_2 =: d$, there exist an isometry $\tau: \R^d \to \R^d$ with $\tau (\Omega_2) = \Omega_1$ and a constant $c \in \C \setminus \{0\}$ such that
\begin{displaymath}
	Uf = cf\circ \tau \qquad \text{for all } f \in C(\overline{\Omega}_1),
\end{displaymath}
and $\beta_2 = \beta_1 \circ \tau|_{\partial\Omega_2}$.
\end{corollary}

\begin{proof}
First observe that $\Delta^{\beta_1}_{\Omega_1}$ and $\Delta^{\beta_2}_{\Omega_2}$ are isospectral, as Proposition~\ref{prop:unitary} shows. Now, as in the proof of Corollary~\ref{cor:dirichlet-c0-kac-disj}, the Weyl asymptotic formula \eqref{eq:weyl}, also valid for the Neumann and Robin Laplacians (see \cite{ivrii80} or \cite{safarov97}, or cf.~\cite[Eq.~(1.7)]{arendt09}), implies that $d_1=d_2$ and $|\Omega_1| = |\Omega_2|$. Moreover, since $U$ is invertible, it satisfies assumption (d) (and (c)), in addition to (a) and (b). We may thus apply Lemma~\ref{lem:neumann-form-1} and Lemma~\ref{lem:neumann-form-2}. But since $|\omega|=|\Omega_1|=|\Omega_2|$ and all sets have Lipschitz boundary, we conclude that $\omega=\Omega_2$.

It remains to show that $\beta_2 = \beta_1 \circ \tau|_{\partial\Omega_2}$. To this end, observe that $\tau$ induces a bijective correspondence between $D(\Delta^{\beta_1}_{\Omega_1})$ and $D(\Delta^{\beta_2}_{\Omega_2})$, $f \in D(\Delta^{\beta_1}_{\Omega_1})$ if and only if $f \circ \tau \in D(\Delta^{\beta_2}_{\Omega_2})$. For any such $f$, we have that $\frac{\partial f}{\partial \nu} + \beta_1 f = 0$ $\sigma$-almost everywhere on $\partial\Omega_1$, whence also
\begin{displaymath}
	\frac{\partial (f\circ \tau)}{\partial \nu} + (\beta_1 \circ \tau) (f \circ \tau) = 0
\end{displaymath}
$\sigma$-almost everywhere on $\partial\Omega_2$. But, by the above correspondence, this says exactly that any $g \in D(\Delta^{\beta_2}_{\Omega_2})$, which by definition satisfies the Robin condition $\frac{\partial g}{\partial \nu} + \beta_2 g = 0$ almost everywhere, also satisfies $\frac{\partial g}{\partial \nu} + (\beta_1 \circ \tau) g = 0$ almost everywhere on $\partial\Omega_2$. This implies that $\beta_1 \circ \tau = \beta_2$ almost everywhere and hence everywhere on $\partial \Omega_2$, since $\beta_1$ and $\beta_2$ were assumed continuous.
\end{proof}

\section{Disjointness-preserving operators intertwining Laplacians on $L^2$}
\label{sec:dirichlet-l2}

We now wish to state similar results for $L^2$-spaces, principally but not exclusively for the Dirichlet Laplacian. Since the techniques involved are somewhat different from the $C_0$-case, we will also need a slightly different set of assumptions. In this section we will write $\Delta^D_\Omega$ for the operator $\DLL$ introduced in Definition~\ref{def:dirichlet-laplacian-l2}, and $\lambda_k^D (\Omega)$, $k\geq 1$, for its eigenvalues. We start with the following definition (see \cite[Sections~1 and~3]{arendt02}).

\begin{definition}
An open set $\Omega \subset \R^d$ is called \emph{regular in capacity} if for all $x \in \partial\Omega$ and $r>0$,
\begin{displaymath}
	\capacity ( \Omega \setminus B(z,r)) > 0,
\end{displaymath}
where the capacity $\capacity (A)$ of a set $A \subset \R^d$ is defined by
\begin{displaymath}
	\capacity (A) = \inf \{ \|u\|_{H^1(\R^d)}^2: u \in H^1(\R^d), \, u\geq 1 \text{ in a neighbourhood of } A \}.
\end{displaymath}
\end{definition}

We have the hierarchy Lipschitz boundary $\implies$ regular in topology $\implies$ regular in measure $\implies$ regular in capacity. Lebesgue's cusp is regular in capacity but not Dirichlet regular.

\begin{theorem}
\label{thm:disj-pres-dirichlet-l2-isom}
Suppose $\Omega_1,\Omega_2 \subset \R^d$ are bounded and connected open sets, which are regular in capacity, and there exists a bounded, linear mapping $0 \neq U: L^2(\Omega_1) \to L^2 (\Omega_2)$ such that $U (D(\Delta^D_{\Omega_1})) \subset D(\Delta^D_{\Omega_2})$, and
\begin{enumerate}
\item[(a)] $f\cdot g= 0$ implies $(Uf)\cdot (Ug) = 0$ for all $f,g \in L^2(\Omega_1)$, and
\item[(b)] $U(\Delta f) = \Delta (Uf)$ for all $f \in D(\Delta^D_{\Omega_1})$.
\end{enumerate}
Suppose in addition that $\lambda_1^D (\Omega_1) = \lambda_1^D (\Omega_2)$. Then there exist an isometry $\tau: \R^d \to \R^d$ with $\tau (\Omega_2) = \Omega_1$ and a constant $c \in \C \setminus \{0\}$ such that
\begin{displaymath}
	Uf = cf \circ \tau \qquad \text{for all } f \in L^2(\Omega_1).
\end{displaymath}
\end{theorem}

By Proposition~\ref{prop:unitary}, the condition $\lambda_1^D(\Omega_1)=\lambda_1^D(\Omega_2)$ is much weaker than invertibility of $U$. The condition of regularity in capacity is optimal for results of this kind. Indeed, to each open set $\Omega \subset \R^d$ there exists a unique open set $\widetilde\Omega \subset \R^d$ which is regular in capacity and such that $\Omega \subset \widetilde\Omega$ and $\capacity (\widetilde\Omega \setminus \Omega) = 0$. This implies that $L^2 (\Omega) = L^2 (\widetilde\Omega)$ and $\Delta^D_{\Omega_1} = \Delta^D_{\Omega_2}$. In particular, as in Theorem~\ref{thm:disj-pres-dirichlet-c0-isom} and Corollary~\ref{cor:dirichlet-c0-kac-disj}, we may drop the assumption of regularity in capacity at the cost of allowing $\tau^{-1}(\Omega_1)$ and $\Omega_2$ to differ by a set of capacity zero. 

The proof of Theorem~\ref{thm:disj-pres-dirichlet-l2-isom} will be based on a more abstract result which will also allow us to treat the Neumann and Robin Laplacians (albeit under stronger assumptions than those of Theorem~\ref{thm:disj-pres-dirichlet-l2-isom}; more precisely, in place of the condition $\lambda_1^D (\Omega_1) = \lambda_1^D (\Omega_2)$ we will require the invertibility of $U$).

The idea of the proof consists showing that the modulus operator $|U|$ of $U$ exists and satisfies the same conditions. Since this operator is positive, we may then apply a previous result of one of the authors \cite[Theorem 2.1]{arendt01} to obtain the conclusion for $|U|$. A theorem due to Zaanen \cite{zaanen75} linking a modulus operator acting via multiplication to the original operator will finally yield the result for $U$. For convenience of reference, we reproduce \cite[Theorem 2.1]{arendt01} here.

\begin{theorem}[\cite{arendt01}, Theorem 2.1]
\label{thm:arendt-thm21}
Suppose $\Omega_1, \Omega_2 \subset \R^d$ are open, connected and regular in capacity. Assume that there exists a linear operator $0 \neq A: L^2(\Omega_1) \to L^2 (\Omega_2)$ satisfying
\begin{enumerate}
\item[(a)] $|Af|=A|f|$ for all $f \in L^2(\Omega_1)$, and
\item[(b)] $A e^{t\Delta^D_{\Omega_1}} = e^{t\Delta^D_{\Omega_2}}A$ for all $t\geq 0$.
\end{enumerate}
Then there exist an isometry $\tau : \R^d \to \R^d$ and a constant $c>0$ such that $\tau(\Omega_2) = \Omega_1$ and $Af=cf\circ\tau$.
\end{theorem}

In order to give the promised abstract result, we first need some preparation. For $j=1,2$ we let $\Omega_j \subset \R^d$ be bounded and connected open sets and we suppose that $A_j$ are self-adjoint operators on $L^2(\Omega_j)$ which are bounded from below and have compact resolvent. We denote by $(\lambda_n (A_j))_{n\in\N}$ the eigenvalues of $A_j$ ordered as an increasing sequence and repeated according to their necessarily finite multiplicities; we thus have $\lim_{n\to\infty} \lambda_n (A_j) = \infty$.

Denote by $S_j$ the semigroup generated by $-A_j$, $j=1,2$. In addition to the above assumptions on $A_j$, we suppose that $S_j$ is \emph{positive} and \emph{irreducible}. This means that $S_j (t) \geq 0$ for all $t \geq 0$, $j=1,2,$, and that $S_j$ does not leave invariant any non-trivial closed ideal $J$ of $L^2(\Omega_j)$. (Here $J$ is called a \emph{closed ideal} of $L^2(\Omega_j)$ if there exists a measurable set $\omega \subset \Omega_j$ such that $J = \{ f \in L^2(\Omega_j): f|_\omega = 0\}$; $J=0$ and $L^2(\Omega)$ are the trivial closed ideals.) A consequence of this assumption is that $\lambda_1 (A_j)$ is a \emph{principal eigenvalue}, that is, that the eigenspace belonging to $\lambda_1 (A_j)$ is one-dimensional, $\lambda_1 (A_j) < \lambda_2 (A_j)$, and that there is a unique eigenfunction $\psi_j \in D(A_j)$ such that $\psi_j > 0$ and $\|\psi_j\|_{L^2(\Omega_j)}=1$, $j=1,2$. Moreover, $\psi_j \gre 0$, by which we mean that $\psi_j(x)>0$ almost everywhere, $j=1,2$. We call $\psi_j$ the \emph{principal eigenfunction} of $A_j$.

Now let $U: L^2(\Omega_1) \to L^2(\Omega_2)$ be a disjointness-preserving operator. Then, as intimated above, there exists a unique positive operator $|U|: L^2(\Omega_1) \to L^2(\Omega_2)$ such that $|U|f=|Uf|$ for all $f\geq 0$. This \emph{modulus operator} $|U|$ is a lattice homomorphism (see \cite{arendt83,arendt01} for these properties).

With this background we can now formulate and prove the following result from which congruence will later follow. See Remark~\ref{rem:intertwining} concerning the following intertwining property \eqref{eq:l2-abstract-intertwining}.

\begin{proposition}
\label{prop:disj-pres-l2-abstract}
For $j=1,2$ let $\Omega_j \R^d$ be a bounded and connected open set and let $A_j$ be a self-adjoint operator on $L^2(\Omega_j)$ which is bounded from below and has compact resolvent. Assume that the semigroup generated by $-A_j$ is positive and irreducible, $j=1,2$, and that $\lambda_1 (A_1) = \lambda_1 (A_2)$. Let $0 \neq U: L^2(\Omega_1) \to L^2(\Omega_2)$ be a disjointness-preserving operator such that
\begin{equation}
\label{eq:l2-abstract-intertwining}
	US_1(t) = S_2(t)U \qquad \text{for all } t\geq 0.
\end{equation}
Then $|U|S_1(t) = S_2(t)|U|$ for all $t \geq 0$. If in addition there exist $c>0$ and an isometry $\tau: \R^d \to \R^d$ such that $\tau(\Omega_2) = \Omega_1$ and
\begin{equation}
\label{eq:modulus-operator-desired-form}
	|U|f = cf\circ\tau \qquad \text{for all } f \in L^2(\Omega_1),
\end{equation}
then there exists a constant $c_1 \in \C$ such that
\begin{displaymath}
	Uf = c_1 f \circ \tau \qquad \text{for all } f \in L^2(\Omega_1).
\end{displaymath}
\end{proposition}

\begin{proof}
1. Let $\psi_1$ be the principal eigenfunction of $A_1$. We claim that $U\psi_1 =: \psi_2$ is an eigenfunction of $A_2$ for the eigenvalue $\lambda_1 (A_2) = \lambda_1 (A_1) =: \lambda_1$. Indeed,
\begin{equation}
\label{eq:semigroup-intertwining-principals}
	S_2(t)\psi_2 = S_2(t)U\psi_1 = US_1(t)\psi_1 = e^{\lambda_1 t}U\psi_1 = e^{\lambda_1 t}\psi_2.
\end{equation}
Assume for a contradiction that $\psi_2=0$; then $|\psi_2| = |U|\psi_1 = 0$. Now let $0 \leq f \in L^2(\Omega_1)$. Since $\psi_1 \gre 0$, we have that $f = \lim_{n\to \infty} f \wedge n\psi_1$, where $f \wedge n\psi_1 := \inf \{f, n\psi_1\}$. But $|U|(f \wedge n\psi_1) \leq n|U|\psi_1 = 0$, whence $|U|f = \lim_{n\to\infty} |U|(f\wedge n\psi_1) = 0$. This means that $|U|=0$, and hence also $U=0$, contradicting our assumptions. This shows that $\psi_2 \neq 0$, and so \eqref{eq:semigroup-intertwining-principals} implies that $\psi_2$ is indeed an eigenfunction for the principal eigenvalue $\lambda_1$ of $A_2$. Replacing $U$ by $c_2 U$ for some $c_2 \in \C$ if necessary, we may assume that $\psi_2$ is the principal eigenfunction of $A_2$; in particular, $\psi_2 \gre 0$.

2. We claim that $|U|S_1(t) = S_2(t)|U|$ for all $t \geq 0$. To see this, we observe that for any $0 \leq f \in L^2(\Omega_1)$ and any $t\geq 0$ we have
\begin{displaymath}
	|U|S_1(t)f = |US_1(t)f| = |S_2(t)Uf| \leq S_2(t)|U|f
\end{displaymath}
since $S_1(t),S_2(t)\geq 0$. Thus the operator $R(t):= S_2(t)|U| - |U|S_1(t)$ is positive, i.e., $R(t)\geq 0$. In order to show that $R(t)=0$, we recall that $\psi_1 \gre 0$ and $\psi_2 = U\psi_1 \gre 0$ are eigenfunctions for $\lambda_1$ and compute, using the self-adjointness of $S_2(t)$,
\begin{displaymath}
	\langle R(t)\psi_1, U\psi_1 \rangle = \langle S_2(t)|U|\psi_1,U\psi_1\rangle-\langle |U|S_1(t)\psi_1,U\psi_1\rangle
	= \langle |U|\psi_1,S_2(t)U\psi_1\rangle - e^{\lambda_1 t} \langle |U|\psi_1,U\psi_1\rangle = 0,
\end{displaymath}
where $\langle\cdot,\cdot\rangle$ is the $L^2$-inner product, in this case on $L^2(\Omega_1)$. This shows that $\langle R(t)\psi_1, U\psi_1 \rangle = 0$, where $U\psi_1 \gre 0$. It follows that $R(t) \psi_1 = 0$. Now we argue as in Step 1 to deduce that $R(t)f = 0$ for all $0 \leq f \in L^2(\Omega_1)$; and so $R(t)=0$, proving the claim.

3. Now assume that \eqref{eq:modulus-operator-desired-form} holds for some isometry $\tau$ for which $\tau(\Omega_2)=\Omega_1$ and some $c>0$. We wish to show that $U$ is of the same form (possibly for a different constant). Define $\Phi: L^2(\Omega_2) \to L^2(\Omega_1)$ by $\Phi f:= f \circ \tau^{-1}$. Then $\Psi:= \Phi \circ U \in \mathcal{L} (L^2(\Omega_1))$, and, by assumption on the form of $|U|$,
\begin{displaymath}
	|\Psi f| = |Uf\circ \tau^{-1}| = |U||f\circ \tau^{-1}| = c|f\circ \tau^{-1} \circ \tau| = c|f|
\end{displaymath}
(in the sense of lattices). Hence $\Psi$ is local, i.e., $\Psi f(x)=0$ for almost every $x \in \{y\in\Omega_1: f(y)=0\}$, for all $f \in L^2(\Omega_1)$. A theorem of Zaanen  (\cite{zaanen75}, see also \cite[Proposition~1.7]{arendt05}) now yields the existence of a function $k \in L^\infty (\Omega_1)$ such that $|k|=c$ almost everywhere and $\Psi f=kf$ for all $f \in L^2(\Omega_1)$. Appealing to the definition of $\Psi$,
\begin{displaymath}
	Uf (\tau^{-1}(x)) = k(x) f(x) \qquad \text{for all } x \in \Omega_1 \text{ and } f \in L^2(\Omega_1),
\end{displaymath}
equivalently,
\begin{displaymath}
	Uf(y) = k(\tau(y)) f(\tau(y)) \qquad \text{for all } y \in \Omega_2 \text{ and } f \in L^2(\Omega_1),
\end{displaymath}
where now $h:=k\circ \tau \in L^\infty(\Omega_2)$ with $|h|=|k|=c$ almost everywhere. Finally, the argument of Lemma~\ref{lem:intertwining-form} applied to $\omega_1=\Omega_1$ and $\omega_2=\Omega_2$ implies that $h$ is constant.
\end{proof}

\begin{proof}[Proof of Theorem~\ref{thm:disj-pres-dirichlet-l2-isom}]
Condition (b) implies that
\begin{displaymath}
	US_1(t) = S_2(t)U
\end{displaymath}
for all $t \geq 0$, where $S_1(t):= e^{t\Delta_{\Omega_1}^D}$, $S_2(t) := e^{t\Delta_{\Omega_2}^D}$. Now clearly both operators $-\Delta_{\Omega_j}^D$, $j=1,2$, are self-adjoint, bounded from below and have compact resolvent; moreover, $S_j(t) \geq 0$ and the semigroups $S_j$ are irreducible (see, e.g., \cite{ouhabaz05}).

Thus by Proposition~\ref{prop:disj-pres-l2-abstract} we have $|U|S_1(t)=S_2(t)|U|$ for all $t\geq 0$. Now Theorem~\ref{thm:arendt-thm21} implies that $|U|$ has the desired form. An application of the second part of Proposition~\ref{prop:disj-pres-l2-abstract} finally shows that $U$ also has the desired form.
\end{proof}

This allows us to give a positive result for disjointness-preserving invertible operators in the nature of Corollary~\ref{cor:dirichlet-c0-kac-disj}, but for operators on $L^2$ in place of $C_0$. We recall that the existence of an invertible intertwining operator on $L^2$ is equivalent to the existence of a unitary intertwining operator, cf.~Proposition~\ref{prop:unitary}.

\begin{corollary}
\label{cor:dirichlet-l2-kac-disj}
Suppose $\Omega_1 \subset \R^{d_1}$ and $\Omega_2 \subset \R^{d_2}$ are two open, bounded and connected sets which are regular in capacity. Suppose also that $U: L^2(\Omega_1) \to L^2(\Omega_2)$ is an invertible disjointness-preserving operator intertwining the Dirichlet Laplacians on $L^2 (\Omega_1)$ and $L^2(\Omega_2)$ in the sense of Definition~\ref{def:intertwining}. Then $d_1=d_2=:d$, and there exist an isometry $\tau: \R^d \to \R^d$ with $\tau (\Omega_2) = \Omega_1$ and a constant $c \in \C \setminus \{0\}$, such that
\begin{displaymath}
	Uf = cf \circ \tau \qquad \text{for all } f \in L^2(\Omega_1).
\end{displaymath}
\end{corollary}

\begin{proof}
By Proposition~\ref{prop:unitary}, the Dirichlet Laplacians on $L^2(\Omega_1)$ and $L^2(\Omega_2)$ have the same spectrum. In particular, $\lambda_1^D (\Omega_1) = \lambda_1^D (\Omega_2)$, and the Weyl asymptotics imply that $d_1=d_2$. Hence we may apply Theorem~\ref{thm:disj-pres-dirichlet-l2-isom}.
\end{proof}

It is remarkable that the analogous result to Theorem~\ref{thm:arendt-thm21} does not hold for manifolds \cite[Example~4.7]{arendt12}, and it is unknown for Neumann or Robin boundary conditions in the Euclidean case. However, it is true for invertible $U$. Thus, while it is unclear whether Theorem~\ref{thm:disj-pres-dirichlet-l2-isom} continues to hold for the Neumann and Robin Laplacians, we can obtain a version of Corollary~\ref{cor:dirichlet-l2-kac-disj} for them. For this we assume that $\Omega_1,\Omega_2 \subset \R^d$ are two bounded Lipschitz domains, take $0\leq \beta \in L^\infty (\partial\Omega_j)$, $j=1,2$, and define $-\Delta_{\Omega_j}^{\beta_j} := -\Delta_{L^2(\Omega_j)}^{\beta_j}$, $j=1,2$, as in Definition~\ref{def:neumann-laplacian-l2}(c) (where we again recall that $\beta_j \equiv 0$ corresponds to the Neumann Laplacian). We also recall that the operators $-\Delta_{\Omega_j}^{\beta_j}$ are self-adjoint and bounded from below, and have compact resolvent; moreover, their semigroups $S_j^{\beta_j}(t):=e^{-t\Delta_{\Omega_j}^{\beta_j}}$ are positive and irreducible (see \cite{ouhabaz05}).

\begin{theorem}
\label{thm:neumann-robin-l2-kac-disj}
Under the above assumptions on $\Omega_j$ and $\beta_j$, $j=1,2$, suppose that there exists an invertible disjointness-preserving operator $U:L^2(\Omega_1) \to L^2(\Omega_2)$ such that
\begin{displaymath}
	S_2^{\beta_2}(t)U = US_1^{\beta_1}(t) \qquad \text{for all } t\geq 0.
\end{displaymath}
Then there exist an isometry $\tau: \R^d \to \R^d$ such that $\tau (\Omega_2) = \Omega_1$ and a constant $c \in \R \setminus \{0\}$ such that
\begin{displaymath}
	Uf = cf\circ \tau \qquad \text{for all } f \in L^2(\Omega_1).
\end{displaymath}
In particular, $\Omega_1$ and $\Omega_2$ are congruent. Moreover, $\beta_2 = \beta_1 \circ \tau|_{\partial\Omega_2}$.
\end{theorem}

\begin{proof}
In view of Proposition~\ref{prop:disj-pres-l2-abstract} and the above remarks, this follows from \cite[Theorem~3.21]{arendt02}.
\end{proof}

\bibliographystyle{amsplain}

\providecommand{\bysame}{\leavevmode\hbox to3em{\hrulefill}\thinspace}
\providecommand{\href}[2]{#2}

\end{document}